\documentclass[CRMATH,Unicode,manuscript]{cedram}

\usepackage{mathtools}
\usepackage{bm}
\usepackage{bbm}
\usepackage{pifont} 
\usepackage{mathrsfs} 

\newcommand{\nc}{\newcommand}

\nc{\symbvec}[1]{\boldsymbol{\bm{#1}}} 

\nc{\bfx}{\mathbf{x}} 
\nc{\bfy}{\mathbf{y}} 
\nc{\bfz}{\mathbf{z}} 
\nc{\bfu}{\mathbf{u}} 
\nc{\bfv}{\mathbf{v}} 
\nc{\bfw}{\mathbf{w}} 
\nc{\bft}{\mathbf{t}} 
\nc{\bfb}{\mathbf{b}} 
\nc{\bfn}{\mathbf{n}} 
\nc{\bff}{\mathbf{f}} 
\nc{\bfq}{\mathbf{q}} %
\nc{\bfd}{\mathbf{d}} 
\nc{\bfe}{\mathbf{e}} 

\nc{\bfA}{\mathbf{A}} 
\nc{\bfR}{\mathbf{R}} 
\nc{\bfD}{\mathbf{D}} 
\nc{\bfI}{\mathbf{I}} 
\nc{\bfM}{\mathbf{M}} 
\nc{\bfK}{\mathbf{K}} 
\nc{\bfV}{\mathbf{V}} 
\nc{\bfW}{\mathbf{W}} 
\nc{\bfC}{\mathbf{C}} 
\nc{\bfP}{\mathbf{P}} 
\nc{\bfZ}{\mathbf{Z}} 
\nc{\bfF}{\mathbf{F}} 
\nc{\bfQ}{\mathbf{Q}} %
\nc{\bfU}{\mathbf{U}} %
\nc{\bfB}{\mathbf{B}} %
\nc{\bfH}{\mathbf{H}} %
\nc{\bfE}{\mathbf{E}} %
\nc{\bfS}{\mathbf{S}} %
\nc{\bfX}{\mathbf{X}} %

\nc{\bfId}{\mathbf{I}_{\mathrm{d}}} 

\nc{\hatu}{\hat{u}} 
\nc{\hatv}{\hat{v}} 

\nc{\bbN}{\mathbb{N}} 
\nc{\bbR}{\mathbb{R}} 
\nc{\bbC}{\mathbb{C}} 
\nc{\bbP}{\mathbb{P}} 
\nc{\bbH}{\mathbb{H}} 
\nc{\bbA}{\mathbb{A}} %
\nc{\bbT}{\mathbb{T}} %
\nc{\bbD}{\mathbb{D}} %
\nc{\bbJ}{\mathbb{J}} %
\nc{\bbF}{\mathbb{F}} %
\nc{\bbM}{\mathbb{M}} %
\nc{\bbQ}{\mathbb{Q}} %
\nc{\bbK}{\mathbb{K}} %

\nc{\wH}{\widetilde{H}}
\nc{\wA}{\widetilde{A}}

\nc{\calS}{\mathcal{S}} 
\nc{\calD}{\mathcal{D}} 
\nc{\calT}{\mathcal{T}} 
\nc{\calR}{\mathcal{R}} 
\nc{\calP}{\mathcal{P}} 
\nc{\calV}{\mathcal{V}} 
\nc{\calW}{\mathcal{W}} 
\nc{\calJ}{\mathcal{J}} 
\nc{\calL}{\mathcal{L}} 
\nc{\calH}{\mathcal{H}} 
\nc{\calO}{\mathcal{O}} 
\nc{\calE}{\mathcal{E}} 
\nc{\calI}{\mathcal{I}} 

\nc{\rmP}{\mathrm{P}}
\nc{\rmI}{\mathrm{I}}
\nc{\rmG}{\mathrm{G}}
\nc{\rmA}{\mathrm{A}}
\nc{\rmR}{\mathrm{R}}
\nc{\rmT}{\mathrm{T}}
\nc{\rmV}{\mathrm{V}}
\nc{\rmW}{\mathrm{W}}
\nc{\rmK}{\mathrm{K}}
\nc{\rmL}{\mathrm{L}}
\nc{\rmH}{\mathrm{H}}
\nc{\rmX}{\mathrm{X}}
\nc{\rmY}{\mathrm{Y}}
\nc{\rmS}{\mathrm{S}}
\nc{\rmD}{\mathrm{D}}

\nc{\fraku}{\mathfrak{u}}
\nc{\frakv}{\mathfrak{v}}
\nc{\frake}{\mathfrak{e}}

\nc\diff{\mathop{}\!\mathrm{d}}   
\DeclareMathOperator{\supp}{supp}   
\DeclareMathOperator{\diag}{diag} 

\DeclareMathOperator{\Id}{Id} 

\nc{\bfPi}{\bm{\Pi}} 
\nc{\bfSigma}{\bm{\Sigma}} %
\DeclareMathOperator{\Div}{div}

\DeclareMathOperator{\GL}{G} 
\DeclareMathOperator{\V}{V} 
\DeclareMathOperator{\K}{K} 
\DeclareMathOperator{\Ktilde}{\tilde{\K}} 
\DeclareMathOperator{\W}{W} 
\DeclareMathOperator{\CP}{P} 
\DeclareMathOperator{\A}{A} %
\DeclareMathOperator{\B}{B} %
\nc{\SumFourier}{\sum_{m \in \mathbb{Z}}} 

\newcommand{\traceNeu}[1][]{\gamma_{\mathrm{N}}^{#1}}
\newcommand{\traceDir}[1][]{\gamma_{\mathrm{D}}^{#1}}
\newcommand{\traceDiv}[1][]{\gamma_{\Div}^{#1}}
\newcommand{\traceNeuC}[1][]{\gamma_{\mathrm{N},c}^{#1}}
\newcommand{\traceDirC}[1][]{\gamma_{\mathrm{D},c}^{#1}}

\newcommand{\BoldtraceNeu}[1][]{\symbvec{\gamma}_{\mathrm{N}}^{#1}}
\newcommand{\BoldtraceDir}[1][]{\symbvec{\gamma}_{\mathrm{D}}^{#1}}
\newcommand{\BoldtraceDiv}[1][]{\symbvec{\gamma}_{\Div}^{#1}}
\newcommand{\BoldtraceNeuC}[1][]{\symbvec{\gamma}_{\mathrm{N},c}^{#1}}
\newcommand{\BoldtraceDirC}[1][]{\symbvec{\gamma}_{\mathrm{D},c}^{#1}}
\newcommand{\BoldtraceDivC}[1][]{\symbvec{\gamma}_{\Div,c}^{#1}}

\nc{\AnsCombined}{Ansatz-Combined}
\nc{\AnsDL}{Ansatz-DL}
\nc{\AnsSL}{Ansatz-SL}
\nc{\DirCombined}{Direct-Combined}
\nc{\DirFK}{Direct-FK}
\nc{\DirSK}{Direct-SK}
\nc{\DirTraceDir}{Direct-TrDir}
\nc{\DirTraceNeu}{Direct-TrNeu}

\nc{\mH}{\mathrm{H}}
\nc{\mX}{\mathrm{X}}
\nc{\setMI}{\mathcal{S}}
\nc{\bx}{\symbvec{x}} 
\nc{\by}{\symbvec{y}} 

\makeatletter
\DeclareFontFamily{OMX}{MnSymbolE}{}
\DeclareSymbolFont{MnLargeSymbols}{OMX}{MnSymbolE}{m}{n}
\SetSymbolFont{MnLargeSymbols}{bold}{OMX}{MnSymbolE}{b}{n}
\DeclareFontShape{OMX}{MnSymbolE}{m}{n}{
    <-6>  MnSymbolE5
   <6-7>  MnSymbolE6
   <7-8>  MnSymbolE7
   <8-9>  MnSymbolE8
   <9-10> MnSymbolE9
  <10-12> MnSymbolE10
  <12->   MnSymbolE12
}{}
\DeclareFontShape{OMX}{MnSymbolE}{b}{n}{
    <-6>  MnSymbolE-Bold5
   <6-7>  MnSymbolE-Bold6
   <7-8>  MnSymbolE-Bold7
   <8-9>  MnSymbolE-Bold8
   <9-10> MnSymbolE-Bold9
  <10-12> MnSymbolE-Bold10
  <12->   MnSymbolE-Bold12
}{}

\let\llangle\@undefined
\let\rrangle\@undefined
\DeclareMathDelimiter{\llangle}{\mathopen}%
                     {MnLargeSymbols}{'164}{MnLargeSymbols}{'164}
\DeclareMathDelimiter{\rrangle}{\mathclose}%
                     {MnLargeSymbols}{'171}{MnLargeSymbols}{'171}
\makeatother
\usepackage[capitalise,noabbrev]{cleveref}

\title{Multi-domain FEM-BEM coupling with several impenetrable obstacles}

\alttitle{Couplage FEM-BEM multi-domaine avec plusieurs obstacles impénétrables}  

\author{\firstname{Antonin} \lastname{Boisneault}\CDRorcid{0000-0001-7986-9048}}
\address{POEMS, CNRS, Inria, ENSTA, Institut Polytechnique de Paris, 91120 Palaiseau, France.}
\addressSameAs{2}{Inria, Unité de Mathématiques Appliquées, ENSTA, Institut Polytechnique de Paris, 91120 Palaiseau, France.}
\email[A. Boisneault]{antonin.boisneault@inria.fr}

\author{\firstname{Marcella} \lastname{Bonazzoli}\CDRorcid{0000-0002-0284-5643}}
\address{Inria, Unité de Mathématiques Appliquées, ENSTA, Institut Polytechnique de Paris, 91120 Palaiseau, France.}
\email[M. Bonazzoli]{marcella.bonazzoli@inria.fr}

\author{\firstname{Xavier} \lastname{Claeys}\CDRorcid{0000-0003-0826-6244}}
\addressSameAs{1}{POEMS, CNRS, Inria, ENSTA, Institut Polytechnique de Paris, 91120 Palaiseau, France.}
\email[X. Claeys]{xavier.claeys@ensta.fr}

\author{\firstname{Pierre} \lastname{Marchand}\CDRorcid{0000-0002-2522-6837}}
\addressSameAs{1}{POEMS, CNRS, Inria, ENSTA, Institut Polytechnique de Paris, 91120 Palaiseau, France.}
\email[P. Marchand]{pierre.marchand@inria.fr}



\altkeywords{couplage FEM-BEM, décomposition de domaine, points de croisement, équation de Helmholtz}

\begin{abstract} 
In~\cite{BoisneaultBonazzoliEtAl2026DFB}, we have analyzed a new formulation of the coupling of finite and boundary element methods (FEM-BEM) for Helmholtz problems, involving several heterogeneous bounded subdomains and one homogeneous unbounded subdomain. This formulation, called Generalized Optimized Schwarz Method (GOSM), is substructured, that is, its unknowns are associated with the subdomains interfaces. 
To derive the GOSM, the first step was to prove that a solution to the Helmholtz problem satisfies a specific multi-domain variational formulation, which involves one operator for each subdomain, each operator being independent of the others. 
In the present contribution, we design a variational formulation of that type for a more general geometrical and material configuration: several heterogeneous bounded subdomains, impenetrable obstacles and homogeneous subdomains are allowed. Note that, like in~\cite{BoisneaultBonazzoliEtAl2026DFB}, we assume that only one subdomain is unbounded, and that its boundary is bounded. The domain partition can have cross-points, that is, points where at least three subdomains are adjacent. We also prove that a solution to the initial Helmholtz problem can be recovered from a solution to the multi-domain variational formulation. 
This shows that the GOSM is a general and flexible framework to model acoustic wave propagation, as it can handle both multi-domain FEM-BEM coupling and (weakly imposed) boundary conditions on several obstacles.
\end{abstract}

\begin{altabstract} 
  Dans~\cite{BoisneaultBonazzoliEtAl2026DFB}, nous avons étudié une nouvelle formulation de type ``méthode des éléments finis - méthode des éléments finis de frontière'' (abrégé FEM-BEM de l'anglais Finite Element Method - Boundary Element Method) pour les problèmes de Helmholtz, la géométrie de ces derniers pouvant être constituée de plusieurs sous-domaines hétérogènes bornés ainsi que d'un sous-domaine homogène non borné. Cette formulation, nommée ``méthode de Schwarz optimisée généralisée'' (abrégé GOSM de l'anglais Generalized Optimized Schwarz Method), est sous-structurée, c'est-à-dire que ses inconnues sont associées aux interfaces entre les sous-domaines. 
  Afin d'établir la GOSM, la première étape consiste à prouver qu'une solution du problème de Helmholtz satisfait une formulation variationnelle multi-domaine spécifique, celle-ci impliquant un opérateur pour chaque sous-domaine, chaque opérateur étant indépendant des autres.
  Dans cette contribution, nous concevons une telle formulation variationnelle pour des géométries et matériaux plus généraux : plusieurs sous-domaines hétérogènes bornés, obstacles impénétrables et sous-domaines homogènes non bornés sont autorisés. Toutefois, tout comme dans~\cite{BoisneaultBonazzoliEtAl2026DFB}, nous supposons qu'un seul sous-domaine est non borné, et que son bord est borné. La partition du domaine peut avoir des points de croisement (aussi appelés jonctions), ces points où au moins trois sous-domaines sont adjacents. Nous prouvons également qu'une solution du problème de Helmholtz initial peut être reconstruite à partir d'une solution de la formulation variationnelle multi-domaine.
  Tout cela montre que la GOSM est un cadre général et flexible pour modéliser la propagation d'ondes acoustiques, puisqu'elle permet de gérer simultanément des couplages FEM-BEM multi-domaines et des conditions aux limites (faiblement imposées) sur plusieurs obstacles.
\end{altabstract}

\begin{document}

\maketitle

\section{Introduction}

To simulate time-harmonic acoustic wave propagation in complex media, modeled by the Helmholtz equation \(\Delta u + \kappa^2 u = 0\), it can be advantageous to use the knowledge of how the wavenumber \(\kappa\) varies in different regions. 
By reformulating local problems in homogeneous subdomains as equations set on their boundary, the problem can be solved by coupling \emph{finite element and boundary element methods} (see e.g.~\cite{JohnsonNedelec1980CBI,BielakMacCamy1983EIP,Costabel1987SMC,HeKaitai1990CMF}). This consists in dealing with heterogeneous regions using classical variational formulations obtained by integration by parts, and with homogeneous ones using \emph{Boundary Integral Equations} (BIEs). 
However, once discretized, such linear systems lead to sums of sparse matrices and densely populated matrices, respectively arising from heterogeneous and homogeneous sub-problems, discretized with the FEM and the BEM\@. This rules out classical preconditioners.

Alternatively, Domain Decomposition (DD) techniques, such as \emph{Optimized Schwarz Methods} (OSMs), are also well-suited to deal with Helmholtz problems with varying wavenumber.
Apart from~\cite{BendaliBoubendirEtAl2007FDD,bonnet:hal-05249078,CaudronAntoineEtAl2020OWC,LangerSteinbach2005CBF,LangerPechstein2007CFS,LangerSteinbach2007CFB}, few studies have investigated DD in the context of FEM-BEM coupling.
Additionally, these studies deal with a configuration with only two subdomains. 
When it comes to multi-domain configurations involving \emph{cross-points} (the points where at least three subdomains are adjacent), even fewer studies exist. For instance, at the discrete level, in several studies~\cite{FarhatLesoinneEtAl2001FDU,BendaliBoubendir2006NOD,BoubendirBendaliEtAl2008CNO} one additional unknown per cross-points is considered. 

We emphasize that, already at the continuous level, the presence of cross-points raises functional analysis issues. 
Indeed, for a given subdomain \(\omega\), the restriction of its boundary to its common part with a neighboring subdomain \(\tilde{\omega}\) is not continuous in usual Sobolev spaces used for Dirichlet or Neumann traces, see e.g.~\cite[§6.2]{MR3202533} for more details.
Hence, the most natural multi-domain formulations that would use such restriction operators can not be written in a proper function space framework. Moreover, cross-points may deteriorate the convergence once the problem is discretized (see e.g.~\cite[§4]{MR4106670}). Therefore, our aim is to design an optimized Schwarz method (OSM) robust to cross-points, with a sound functional analysis framework and convergence theory, for general propagation media.

In~\cite{BonazzoliClaeys2024MFB}, the authors have designed new multi-domain FEM-BEM formulations allowing for a clean treatment of cross-points, but their discretization and practical implementation are not straightforward. 
A substructured formulation set on the skeleton of the domain partition has been derived in~\cite{FlorianHiptmairEtAl2023SIE,GraessleHiptmairEtAl2025SSI}, which only relies on BIEs, even for heterogeneous subdomains. A discrete study is performed in~\cite{GraessleSauter2026BVC}, but the practical implementation is not straightforward.
By adopting instead a substructuring DD approach, an alternative formulation, called Generalized OSM (GOSM), has been introduced in~\cite{MR4433119, Claeys2023NOS} for \emph{bounded} geometries. A rigorous theory guarantees geometric convergence of the GOSM regardless of the presence of cross-points. Recently, in~\cite{BoisneaultBonazzoliEtAl2026DFB} we have extended the GOSM to FEM-BEM coupling to treat also an unbounded subdomain (with bounded boundary).

The GOSM relies on a few assumptions, one of them being that a solution of the initial Helmholtz problem must satisfy a specific multi-domain variational formulation, which involves one operator for each subdomain, each operator being independent of the others. In the present contribution, we establish such a formulation for general propagation media, that is, any number of heterogeneous bounded subdomains, impenetrable obstacles and homogeneous subdomains are considered. In particular, this formulation involves the Costabel FEM-BEM coupling~\cite{Costabel1987SMC} for the heterogeneous and homogeneous subdomains, and weakly imposed boundary conditions on each impenetrable obstacle. 
Given this new formulation, the whole theory for the GOSM in~\cite{BoisneaultBonazzoliEtAl2026DFB} can be seamlessly extended to such general configurations. 

The paper is organized as follows. In Section~\ref{sec:tp} we define the Helmholtz problem we aim to solve, as well as the domain partition. After introducing several notations, we rewrite equivalently the initial problem as a transmission problem.
We recall classical results about boundary integral operators in Section~\ref{sec:bios}. 
Suitable multi-domain function spaces are defined in Section~\ref{sec:multitrace_spaces}. In particular, the so-called single-trace spaces allow us expressing transmission conditions on the subdomains interfaces. 
Finally, in Section~\ref{sec:varf}, we state the sought variational formulation, and establish its equivalence with the initial transmission problem. We also highlight that the equivalence relies on the use of the Costabel FEM-BEM coupling.

\section{Definition of the transmission problem}\label{sec:tp}

We consider a domain \(\Omega_{\mathrm{O}} \subset \mathbb{R}^d\) (\(d=2,3\)), which represents impenetrable obstacles, not necessarily connected nor bounded but with bounded Lipschitz boundary, such that \(\Omega_{\mathrm{O}}^c \coloneqq \mathbb{R}^d \setminus \overline{\Omega_{\mathrm{O}}}\) is connected. We are interested in solving the Helmholtz problem: find \(u\) in \(H^1_{\mathrm{loc}}(\Delta, \overline{\Omega_{\mathrm{O}}^c})\) such that
\begin{equation}\label{pbm:initialPb}
  \begin{cases}
      - \Delta u - \kappa(x)^2 u = f \text{ in } \Omega_{\mathrm{O}}^c,\\
        +\;\text{Boundary condition on }\Gamma_{\mathrm{O}} \coloneqq \partial\Omega_{\mathrm{O}}, \\
        +\;\text{Sommerfeld's radiation condition if } \Omega_{\mathrm{O}} \text{ bounded,}
      \end{cases}
\end{equation}
where \(\kappa:\Omega_{\mathrm{O}}^c \to \mathbb{R}^{+*}\) is the wavenumber, the source term \(f \in L^2(\Omega_{\mathrm{O}}^c)\) has bounded support, and boundary conditions can be of Dirichlet (sound-soft), Neumann (sound-hard), Robin, or even mixed type. We refer e.g.~to~\cite[Def.~2.6.1]{SauterSchwab2011BEM} for the definition of \(H^1_{\mathrm{loc}}(\Delta, \overline{\Omega_{\mathrm{O}}^c})\). We aim at writing equivalently Problem~\eqref{pbm:initialPb} as a transmission problem, by distinguishing homogeneous and heterogeneous subdomains, and introduce several notations for this purpose.

\paragraph{Domain partition} 
We assume that \(\Omega_{\mathrm{O}}\) admits \(N_{\mathrm{O}} \in \mathbb{N}\) non-overlapping connected subdomains \(\Omega_{\mathrm{O},j}\) (\(j=1,\dots, N_{\mathrm{O}}\)) satisfying \(\partial\Omega_{\mathrm{O},j} \subset \Gamma_{\mathrm{O}}\). 
This can also be written \( \overline{\Omega_{\mathrm{O}}} =  \bigcup_{j=1}^{N_{\mathrm{O}}} \overline{\Omega_{\mathrm{O},j}}\) and \(\Gamma_{\mathrm{O}} = \bigcup_{j=1}^{N_{\mathrm{O}}} \partial\Omega_{\mathrm{O},j}\). 
Because \(\Omega_{\mathrm{O}}\) is a Lipschitz domain, the condition \(\partial\Omega_{\mathrm{O},j} \subset \Gamma_{\mathrm{O}}\) implies that the intersection of two distinct \(\overline{\Omega_{\mathrm{O},j}}\) is empty. 
Note that the boundary of any \(\Omega_{\mathrm{O},j}\) is bounded. 

We also assume that \(\Omega_{\mathrm{O}}^c\) is partitioned into non-overlapping connected subdomains with bounded boundary that match the variation of \(\kappa\): 
\[
\overline{\Omega_{\mathrm{O}}^c} = \left(\bigcup_{k=1}^{N_{\mathrm{B}}} \overline{\Omega_{\mathrm{B},k}}\right) \cup \left(\bigcup_{l=1}^{N_{\mathrm{F}}} \overline{\Omega_{\mathrm{F},l}}\right),
\]
with \(\kappa(x) = \kappa_k>0\) in \(N_{\mathrm{B}} \in \mathbb{N}\) homogeneous subdomains \(\Omega_{\mathrm{B},k}\) (\(k=1,\dots, N_{\mathrm{B}}\)) (where BEM can be applied), and \(\kappa\) non-constant in \(N_{\mathrm{F}} \in \mathbb{N}\) heterogeneous bounded subdomains \(\Omega_{\mathrm{F},l}\) (\(l=1,\dots, N_{\mathrm{F}}\)) (where FEM is used instead). One of the \(\Omega_{\mathrm{B},k}\) (say \(\Omega_{\mathrm{B},1}\)) is unbounded if \(\Omega_{\mathrm{O}}\) is bounded (see Figure~\ref{fig:example_geometries}, left). Note that the opposite case, where one of the obstacles \(\Omega_{\mathrm{O},j}\) (say \(\Omega_{\mathrm{O},1}\)) is unbounded, corresponds to acoustic wave propagation in a bounded geometry, that is, Problem~\eqref{pbm:initialPb} with \(\Omega_{\mathrm{O}}^c\) bounded (see Figure~\ref{fig:example_geometries}, right).  
Moreover, \(N_{\mathrm{B}}+N_{\mathrm{F}}\geq 1\) and \(N_{\mathrm{O}}+N_{\mathrm{B}}+N_{\mathrm{F}} \geq 2\). All the subdomains are Lipschitz. For any subdomain \(\omega\) we define the outgoing normal \(\symbvec{n}_{\omega}\) on its boundary \(\Gamma_{\omega} \coloneqq \partial \omega\).

\begin{figure}[!h]
	\centering
	\includegraphics{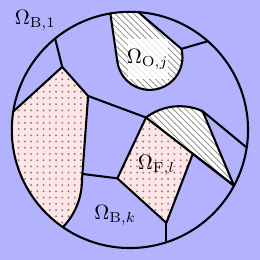}
	\hspace{1.5cm}
	\includegraphics{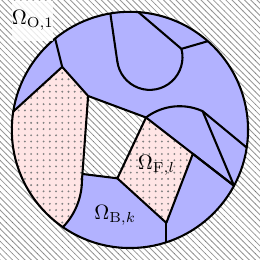}
	\caption[Example of geometries]{
	Examples of an unbounded geometry (left) and a bounded geometry (right). Homogeneous subdomains have a plain violet background, the heterogeneous subdomains have a dotted rose background and the impenetrable subdomains are hatched. 
	Subdomains shape is arbitrary, they are only assumed to be at least Lipschitz.}\label{fig:example_geometries}
\end{figure}

\paragraph{Notation conventions}
  Many indexes will be used throughout this paper, so we emphasize the following conventions that we apply to make the notation more compact.
  \begin{enumerate}
    \item A small (resp.~capital) bold character represents a vector (resp.~a block diagonal operator), and its components are bold if they are themselves vectors.
    \item The indexes \(j,k,l\) are always associated with the mathematical objects indexed respectively by \(\mathrm{O}, \mathrm{B}, \mathrm{F}\). When the same index is involved twice in a mathematical expression, the range to which the index belongs is omitted. 
    For instance, the domain decomposition of \(\Omega_{\mathrm{O}}^c\) can now be written \(\overline{\Omega_{\mathrm{O}}^c} = \left(\bigcup_{k} \overline{\Omega_{\mathrm{B},k}}\right) \cup \left(\bigcup_{l} \overline{\Omega_{\mathrm{F},l}}\right)\).
    In the same way, a vector \(\symbvec{p} \coloneqq \left(p_{1}, \dots, p_{N_{\mathrm{F}}}\right)\) is shortened to \( \left(p_{l}\right)_l \). We emphasize that \(0 \in \mathbb{N}\). Thus, if \(N_{\mathrm{F}} = 0\), \(\Omega_{\mathrm{O}}^c\) is only composed of \(\left(\Omega_{\mathrm{B},k}\right)_k\), and the vector \(\symbvec{p}\) is empty.
    \item When nothing is specified, a bullet point \(\bullet\) as an index always refers to one of the three letters \(\mathrm{O},\mathrm{B},\mathrm{F}\). Additionally, a bold character indexed by a \(\bullet\) is a vector whose size is a multiple of \(N_{\bullet}\) elements that should be clear from the context. For instance, \(\symbvec{p}_{\mathrm{F}} = \left(p_{\mathrm{F},l}\right)_l\) has \(N_{\mathrm{F}}\) components, while \(\symbvec{v}_{\mathrm{F}} = \left(\psi_{\mathrm{F},l}, q_{\mathrm{F},l}\right)_l\) has \(2 N_{\mathrm{F}}\) components.
    \item A Greek letter as an index always refers to a multi-index \((\bullet, m)\), where \(m\) is the index associated with \(\bullet\). By defining the multi-indexes sets \(\setMI_{\bullet} \coloneqq \{\left(\bullet,m\right)\}_{m}\), the domain decomposition of \(\Omega_{\mathrm{O}}^c\) can now be written \(\overline{\Omega_{\mathrm{O}}^c} = \bigcup_{\alpha \in \setMI_{\mathrm{B}}\cup \setMI_{\mathrm{F}}} \overline{\Omega_{\alpha}}\). Additionally, for a vector \(\symbvec{p}_{\bullet} \coloneqq \left(p_{\alpha}\right)_{\alpha \in \setMI_{\bullet}}\), we shorten its expression to \(\symbvec{p}_{\bullet} = \left(p_{\alpha}\right)_{\setMI_{\bullet}}\).
  \end{enumerate}


\paragraph{Traces operators}
For a given domain \(\omega \subset \mathbb{R}^d\), the \emph{interior Dirichlet-Neumann trace operator} on \(\Gamma_{\omega}\), \(\gamma^{\omega} \colon H^1_{\mathrm{loc}}(\Delta, \overline{\omega}) \to H^{{1}/{2}}(\Gamma_{\omega})\times H^{-{1}/{2}}(\Gamma_{\omega})\), is defined as the unique bounded linear operator satisfying \(\gamma^{\omega}(\varphi) \coloneqq \bigl(\traceDir[\omega](\varphi), \traceNeu[\omega](\varphi)\bigr) \coloneqq \bigl(\varphi|_{\Gamma_{\omega}}, \symbvec{n}_{\omega} \cdot \nabla \varphi|_{\Gamma_{\omega}}\bigr)\), for all \(\varphi\in \mathscr{C}^{\infty}(\overline{\omega}) \coloneq\{\varphi\vert_{\omega},\;\varphi\in\mathscr{C}^{\infty}(\mathbb{R}^d)\}\). 
The \emph{exterior (or complementary) Dirichlet-Neumann trace operator} \(\gamma_c^{\omega}\) is similarly defined as the unique bounded linear operator satisfying \(\gamma_c^{\omega}(\varphi) \coloneqq \bigl(\traceDirC[\omega](\varphi), \traceNeuC[\omega](\varphi)\bigr) \coloneqq \bigl(\varphi|_{\Gamma_{\omega}}, \symbvec{n}_{\omega} \cdot \nabla \varphi|_{\Gamma_{\omega}}\bigr)\) for all \(\varphi\in \mathscr{C}^{\infty}(\mathbb{R}^d\setminus\omega)\).
Define also the \emph{interior normal trace map} \(\gamma_{\Div}^{\omega} \colon H_{\mathrm{loc}}(\Div, \overline{\omega}) \to H^{-{1}/{2}}(\Gamma_{\omega})\) as the unique bounded linear operator satisfying \(\gamma_{\Div}^{\omega}(\varphi) \coloneqq \varphi |_{\Gamma_{\omega}} \cdot \symbvec{n}_{\omega}\) for all \(\varphi\in \mathscr{C}^{\infty}(\mathbb{R}^d\setminus\omega)^d\). The \emph{exterior (or complementary) normal trace map} \(\gamma_{\Div,c}^{\omega} \) is similarly defined, taking the exterior trace instead.
For ease of reading, for any domain \(\Omega_{\alpha}\) its boundary \(\Gamma_{\Omega_{\alpha}}\) is denoted \(\Gamma_\alpha\), and the operators \(\traceDir[\Omega_{\alpha}]\), \(\traceNeu[\Omega_{\alpha}]\) and \(\traceDiv[\Omega_{\alpha}]\) are denoted \(\traceDir[\alpha]\), \(\traceNeu[\alpha]\) and \(\traceDiv[\alpha]\). The same conventions apply for the exterior traces.

Let \(\Omega \subset \mathbb{R}^d\) be another domain such that \(\omega \subset \Omega\). For \(u \in H^1_{\mathrm{loc}}(\overline{\Omega})\), the Dirichlet trace of \(u\) on \(\Gamma_{\omega}\) is written indifferently \(\traceDir[\omega](u |_{\omega})\) or \(\traceDir[\omega](u)\).
The vector of Dirichlet traces of any function \(u \in H^1_{\mathrm{loc}}(\overline{\Omega_{\bullet}})\) is denoted \( \symbvec{\traceDir[\bullet]}(u) \coloneqq \bigl( \traceDir[\alpha](u) \bigr)_{\setMI_{\bullet}}\).
Additionally, the Dirichlet trace of a vector \(\symbvec{v}_{\bullet}\) in \(\mathbb{H}^1(\Omega_{\mathrm{\bullet}}) \coloneqq \Pi_{\alpha \in \setMI_{\bullet}} H^1_{\mathrm{loc}}(\overline{\Omega_{\alpha}})\), is defined componentwise by 
\( \BoldtraceDir[\bullet](\symbvec{v}_{\bullet}) \coloneqq \bigl( \traceDir[\alpha](v_{\alpha}) \bigr)_{\setMI_{\bullet}}\).
Similar definitions hold for the Neumann trace, the Dirichlet-Neumann trace pair, and the exterior traces.


\paragraph{Equivalent transmission problem}
Problem~\eqref{pbm:initialPb} can be now rewritten taking into account the domain decomposition: 
find \(u\) in \(H^1_{\mathrm{loc}}(\Delta, \overline{\Omega_{\mathrm{O}}^c})\) such that
\begin{equation}\label{pbm:multi_domain_config}
  \begin{cases}
    \begin{aligned}
      - \Delta u|_{\Omega_{\alpha}} - \kappa|_{\Omega_{\alpha}}^2 u|_{\Omega_{\alpha}} &= f|_{\Omega_{\alpha}} & &\text{ in } \Omega_{\alpha}, & \alpha \in \setMI_{\mathrm{B}} \cup \setMI_{\mathrm{F}},& \\[0.1cm]
      \traceDir[\alpha](u) - \traceDir[\beta](u) &= 0 & &\text{ on } \Gamma_{\alpha} \cap \Gamma_{\beta}, & \alpha,\beta \in \setMI_{\mathrm{B}} \cup \setMI_{\mathrm{F}},& \;\; \alpha \neq \beta \\
      \traceNeu[\alpha](u) + \traceNeu[\beta](u)  &= 0 & &\text{ on } \Gamma_{\alpha} \cap \Gamma_{\beta}, & \alpha,\beta \in \setMI_{\mathrm{B}} \cup \setMI_{\mathrm{F}},& \;\; \alpha \neq \beta \\[0.1cm]
      +\;\text{Boundary condi}&\text{tions} & &\text{ on }\Gamma_{\alpha} \cap \Gamma_{\delta}, & \alpha \in \setMI_{\mathrm{B}} \cup \setMI_{\mathrm{F}},&\;\; \delta \in \setMI_{\mathrm{O}}\\
      &&&&&\hspace*{-8.265cm}+\:\!\text{Sommerfeld's radiation} \text{ condition} \text{ if } \Omega_{\mathrm{O}} \text{ bounded}.
    \end{aligned}
  \end{cases}
\end{equation}
For simplicity, throughout the paper we will consider Dirichlet boundary conditions on all the obstacles, i.e., we impose \(\BoldtraceDirC[\mathrm{O}](u) = \symbvec{g}_{\mathrm{D}}\) for \(\symbvec{g}_{\mathrm{D}} \in \Pi_{j}  H^{1/2}(\Gamma_{\mathrm{O},j})\).
However, we emphasize that the results stated afterwards hold if different boundary conditions are applied on different obstacles, whether they are of Dirichlet, Neumann, Robin or even mixed Dirichlet-Neumann type. 
For simplicity, we also consider that the source term is restricted to the heterogeneous domain, that is to say \(f|_{\Omega_{\mathrm{B}}} = 0\). We denote \(\symbvec{f}_{\mathrm{F}} \coloneqq (f|_{\Omega_{\alpha}})_{\setMI_{\mathrm{F}}}\).

\section{Boundary integral operators}\label{sec:bios}

For \(k=1,\dots, N_{\mathrm{B}}\), denote \(x\mapsto \mathscr{G}_{\kappa_k}(x)\) the outgoing \emph{Green kernel} (or fundamental solution) for the Helmholtz operator with constant wavenumber \(\kappa_k>0\), satisfying \((\Delta + \kappa_k^2) (\mathscr{G}_{\kappa_k})(x) = \delta(x)\) in the sense of distributions, where \(\delta\) is the Dirac delta function. 
The Green kernel is given by \(\mathscr{G}_{\kappa_k}(x)\coloneqq \exp(\imath \kappa_k\vert x\vert)/(4\pi\vert x\vert)\) for \(d = 3\), and \(\mathscr{G}_{\kappa_k}(x)\coloneqq \imath H^{(1)}_{0}(\kappa_k\vert x\vert)/(4\pi)\) for \(d = 2\), where \(z\mapsto H^{(1)}_{0}(z)\) is the \(0\)-th order Hankel function of the first kind, see e.g.~\cite[Chap.10]{MR2723248}.
For any \(x\in \mathbb{R}^d\setminus\Gamma_{\mathrm{B},k}\), and sufficiently smooth traces $(v,p)$, define the \emph{total layer potential operator} by 
\begin{equation*}\label{LayerPotentialOperator}
  \GL_{\mathrm{B},k}(v,p)(x)\coloneqq
  \int_{\Gamma_{\mathrm{B},k}} \left( \symbvec{n}_{\mathrm{B},k} (y)\cdot(\nabla \mathscr{G}_{\kappa_k})(x-y)\, v(y)
  +\;\mathscr{G}_{\kappa_k}(x-y)\, p(y) \right) \diff s (y),
\end{equation*}
where \(\diff s\) refers to the Lebesgue surface measure on \(\Gamma_{\mathrm{B},k}\). 
The map \((v,p)\mapsto \GL_{\mathrm{B},k}(v,p)\vert_{\Omega_{\mathrm{B},k}}\) can be extended by density as a bounded linear operator \(H^{1/2}(\Gamma_{\mathrm{B},k})\times H^{-1/2}(\Gamma_{\mathrm{B},k}) \to H^{1}_{\mathrm{loc}}(\Delta,\overline{\Omega_{\mathrm{B},k}})\). 
For any pair \((v,p)\in H^{1/2}(\Gamma_{\mathrm{B},k})\times H^{-1/2}(\Gamma_{\mathrm{B},k})\), the function \(u \coloneqq \GL_{\mathrm{B},k}(v,p)\) is solution to the Helmholtz equation with wavenumber \(\kappa_k\) in \(\mathbb{R}^d \setminus \Gamma_{\mathrm{B},k}\), see e.g.~\cite[§2.4]{CoKr:bookIEM:1983}. Moreover, if \(\Omega_{\mathrm{B},k}\) is unbounded, the Sommerfeld's radiation condition is satisfied, see~\cite{Sommerfeld1912,MR29463,MR103705,MR1822275}.

Using the interior Dirichlet-Neumann trace map, we form the Calderón projector \(\CP_{\mathrm{B},k} \coloneqq \gamma^{\mathrm{B},k}\cdot\GL_{\mathrm{B},k}\colon H^{1/2}(\Gamma_{\mathrm{B},k})\times H^{-1/2}(\Gamma_{\mathrm{B},k}) \to H^{1/2}(\Gamma_{\mathrm{B},k}) \times H^{-1/2}(\Gamma_{\mathrm{B},k})\), which is commonly split as
\begin{equation*}\label{BIOP}
  \CP_{\mathrm{B},k} =
  \dfrac{1}{2}\begin{bmatrix}
    \Id & 0\\
    0 & \Id
  \end{bmatrix}
  +
  \begin{bmatrix}
    \K_{\mathrm{B},k} & \V_{\mathrm{B},k}\\
    \W_{\mathrm{B},k} & \Ktilde_{\mathrm{B},k}
  \end{bmatrix},
\end{equation*}
where the four classical Boundary Integral Operators (BIOs) appear: \(\K_{\mathrm{B},k}\colon H^{1/2}(\Gamma_{\mathrm{B},k})\to H^{1/2}(\Gamma_{\mathrm{B},k})\) (\emph{double layer}), \(\V_{\mathrm{B},k}\colon H^{-1/2}(\Gamma_{\mathrm{B},k})\to H^{1/2}(\Gamma_{\mathrm{B},k})\) (\emph{single layer}), \(\W_{\mathrm{B},k}\colon H^{1/2}(\Gamma_{\mathrm{B},k})\to H^{-1/2}(\Gamma_{\mathrm{B},k})\) (\emph{hypersingular}), and \(\Ktilde_{\mathrm{B},k}\colon H^{-1/2}(\Gamma_{\mathrm{B},k})\to \mH^{-1/2}(\Gamma_{\mathrm{B},k})\) (\emph{adjoint double layer}). Both the single layer and hypersingular operators are symmetric, while \(\K_{\mathrm{B},k}^* = - \Ktilde_{\mathrm{B},k}\). 
Besides, those operators are involved in the so-called \emph{Calderón equations} (see e.g.~\cite[Sect.~3.4]{SauterSchwab2011BEM}):
\begin{subequations}
  \begin{align}
    \left(\Id/2 - \K_{\mathrm{B},k} \right) \traceDir[\mathrm{B},k](u)      &= \V_{\mathrm{B},k} \traceNeu[\mathrm{B},k](u),  \label{eq:calderon_1} \\
    \left(\Id/2 - \Ktilde_{\mathrm{B},k} \right) \traceNeu[\mathrm{B},k](u) &= \W_{\mathrm{B},k}  \traceDir[\mathrm{B},k](u), \label{eq:calderon_2}
  \end{align}
\end{subequations}
where \(u \in H^1_{\mathrm{loc}}(\overline{\Omega_{\mathrm{B},k}})\) satisfies \(-\Delta u - \kappa_k^2 u = 0\) in \(\Omega_{\mathrm{B},k}\), and Sommerfeld's radiation condition if \(\Omega_{\mathrm{B},k}\) is unbounded.

\section{Multi-trace spaces}\label{sec:multitrace_spaces}

We aim to express equivalently Problem~\eqref{pbm:multi_domain_config} with a variational formulation similar to the one stated in~\cite[Section~9]{BoisneaultBonazzoliEtAl2026DFB}. To this end, we define the \emph{skeleton} \(\Sigma\) as the union of all subdomain interfaces:
\[
  \Sigma \coloneqq \Gamma_{\mathrm{O}} \cup \Gamma_{\mathrm{B}} \cup \Gamma_{\mathrm{F}}, \quad \text{with } \Gamma_\bullet \coloneqq \textstyle\bigcup_{\alpha \in \setMI_{\bullet}} \Gamma_{\alpha}.  
\]
Then, we introduce several multi-domain trace spaces on \(\Sigma\).
The \emph{multi-trace space} is defined by
\begin{equation*}
  \mathbb{H}(\Sigma) \coloneqq \mathbb{H}(\Gamma_{\mathrm{O}}) \times \mathbb{H}(\Gamma_{\mathrm{B}}) \times \mathbb{H}(\Gamma_{\mathrm{F}}),
\end{equation*}
where \(\mathbb{H}(\Gamma_{\mathrm{O}}) \coloneqq \Pi_{j} \mathbb{H}(\Gamma_{\mathrm{O},j})\), \(\mathbb{H}(\Gamma_{\mathrm{B}}) \coloneqq \Pi_{k} \mathbb{H}(\Gamma_{\mathrm{B},k})\) and \(\mathbb{H}(\Gamma_{\mathrm{F}}) \coloneqq \Pi_{l} \mathbb{H}(\Gamma_{\mathrm{F},l})\) are respectively the multi-trace spaces of obstacles, BEM and FEM subdomains, with \(\mathbb{H}(\Gamma_{\bullet, m}) \coloneqq H^{1/2}(\Gamma_{\bullet, m}) \times H^{-1/2}(\Gamma_{\bullet, m})\).
The space \(\mathbb{H}(\Sigma)\) is nothing more than the Cartesian product of local spaces of Dirichlet-Neumann traces.
We also define the \emph{Dirichlet} (resp.~\emph{Neumann}) \emph{multi-trace space}:
\[
\mathbb{H}_{\mathrm{D}}(\Sigma) \coloneqq \mathbb{H}_{\mathrm{D}}(\Gamma_{\mathrm{O}}) \times \mathbb{H}_{\mathrm{D}}(\Gamma_{\mathrm{B}}) \times \mathbb{H}_{\mathrm{D}}(\Gamma_{\mathrm{F}}), \quad
\mathbb{H}_{\mathrm{N}}(\Sigma) \coloneqq \mathbb{H}_{\mathrm{N}}(\Gamma_{\mathrm{O}}) \times \mathbb{H}_{\mathrm{N}}(\Gamma_{\mathrm{B}}) \times \mathbb{H}_{\mathrm{N}}(\Gamma_{\mathrm{F}}), 
\]
with \(\mathbb{H}_{\mathrm{D}}(\Gamma_{\bullet}) \coloneqq \Pi_{\alpha \in \setMI_{\bullet}} H^{1/2}(\Gamma_{\alpha})\),
\(\mathbb{H}_{\mathrm{N}}(\Gamma_{\bullet}) \coloneqq \Pi_{\alpha \in \setMI_{\bullet}} H^{-1/2}(\Gamma_{\alpha})\).  
For \((\symbvec{p},\symbvec{u})\in \mathbb{H}_{\mathrm{N}}(\Sigma) \times \mathbb{H}_{\mathrm{D}}(\Sigma)\), with \(\symbvec{p}=(\symbvec{p}_{\mathrm{O}}, \symbvec{p}_{\mathrm{B}}, \symbvec{p}_{\mathrm{F}})\) and \(\symbvec{u} = (\symbvec{u}_{\mathrm{O}}, \symbvec{u}_{\mathrm{B}}, \symbvec{u}_{\mathrm{F}})\), we define the inner product
\[
\llangle \symbvec{p}, \symbvec{u} \rrangle \coloneqq \textstyle\sum_{\bullet} \langle \symbvec{p}_{\bullet}, \symbvec{u}_{\bullet} \rangle,  
\quad \text{with } 
\langle \symbvec{p}_{\bullet}, \symbvec{u}_{\bullet} \rangle \coloneqq \textstyle\sum_{\alpha \in \setMI_{\bullet}}\langle p_{\alpha}, u_{\alpha}\rangle,
\]
where \(\langle \cdot,\cdot \rangle\) is the generic canonical duality pairing between a Banach or Hilbert space and its dual. More details about this pairing are given in the next paragraph.
Following~\cite{ClaeysHiptmair2014IEA}, we also introduce the \emph{Dirichlet} (resp.~\emph{Neumann}) \emph{single-trace space}, whose tuples satisfy Dirichlet (resp.~Neumann) transmission conditions on the skeleton \(\Sigma\):
\begin{align*}
    &\mathbb{X}_{\mathrm{D}}(\Sigma) \coloneqq 
    \left\{
        (\symbvec{u}_{\mathrm{O}}, \symbvec{u}_{\mathrm{B}}, \symbvec{u}_{\mathrm{F}}) \in \mathbb{H}_{\mathrm{D}}(\Sigma) \mid \exists u^* \in H^1(\mathbb{R}^d) \text{ such that } 
        \BoldtraceDir[\bullet](u^*) =  \symbvec{u}_{\bullet} 
    \right\} \\
    &\mathbb{X}_{\mathrm{N}}(\Sigma) \coloneqq 
    \left\{
        (\symbvec{p}_{\mathrm{O}}, \symbvec{p}_{\mathrm{B}}, \symbvec{p}_{\mathrm{F}}) \in \mathbb{H}_{\mathrm{N}}(\Sigma) \mid \exists p^* \in H(\Div,\mathbb{R}^d) \text{ such that } 
        \BoldtraceDiv[\bullet](p^*) = \symbvec{p}_{\bullet} 
    \right\}
\end{align*}

\paragraph{A characterization of the single-trace spaces}

The \emph{annihilator} or \emph{polar set} (see, e.g.~\cite[§4.6]{Rudin1991book},~\cite[§1.3]{Brezis2011FAS} and~\cite[§4.1.4--4.1.5]{BoffiBrezziEtAl2013MFE}) of any subset \(\mX\) of a Banach space \( \mH\) is defined by
\begin{equation*}\label{DefPolarSet}
  \mX^{\circ} \coloneq\{ \varphi\in \mH^*,\langle \varphi,v\rangle = 0\; \text{ for all } v\in \mX\},
\end{equation*}
where the canonical duality pairing between \(H\) and its topological dual \(H^*\) is denoted \(\langle \cdot,\cdot\rangle \colon \mH^{*}\times\mH\to \mathbb{C}\) and defined by $\langle \varphi, v\rangle\coloneq \varphi(v)$. We emphasize that the duality pairings we consider do \emph{not} involve any complex conjugation, allowing us to write indifferently \(\langle v,\varphi\rangle = \langle \varphi, v\rangle\), for \( v\in\mH, \varphi\in \mH^*\).
The next proposition established that each single-trace space defined above can be characterized as the annihilator of the other. We refer to~\cite[Proposition~2.1]{Claeys2011STI} and~\cite[Theorem~3.1]{ClaeysHiptmairEtAl2013MBI} for a proof.

\begin{proposition}\label{prop:annihilators}
  \(\mathbb{X}_{\mathrm{N}}(\Sigma) = \mathbb{X}_{\mathrm{D}}(\Sigma)^{\circ}\), i.e.\
  \(\symbvec{p} \in \mathbb{X}_{\mathrm{N}}(\Sigma) \iff \llangle \symbvec{p}, \symbvec{v} \rrangle = 0, \ \forall \symbvec{v} \in \mathbb{X}_{\mathrm{D}}(\Sigma)\).

  \hspace{0.49cm}Similarly, \(\mathbb{X}_{\mathrm{D}}(\Sigma) = \mathbb{X}_{\mathrm{N}}(\Sigma)^{\circ}\), i.e.\ \(\symbvec{u} \in \mathbb{X}_{\mathrm{D}}(\Sigma) \iff \llangle \symbvec{q}, \symbvec{u} \rrangle = 0, \ \forall \symbvec{q} \in \mathbb{X}_{\mathrm{N}}(\Sigma)\).
\end{proposition}

Using the single-trace spaces, a characterization of the Dirichlet and Neumann traces on the obstacle \(\Omega_{\mathrm{O}}\) is stated in the next two propositions.
\begin{proposition}\label{prop:dirichlet_annihilator}
  Let \(\Phi \in L_{\mathrm{loc}}^2(\overline{\Omega_{\mathrm{O}}^c})^d\) such that \(\Phi |_{\Omega_{\alpha}} \in H(\Div,\Omega_{\alpha})\) and \(\Phi |_{\Omega_{\beta}} \in H_{\mathrm{loc}}(\Div,\overline{\Omega_{\beta}})\), for all \(\alpha \in \setMI_{\mathrm{F}}\) and \(\beta \in \setMI_{\mathrm{B}}\).  
  Let \(\symbvec{p}_{\mathrm{O}} \in \mathbb{H}_{\mathrm{N}}(\Gamma_{\mathrm{O}})\) and define \(\symbvec{p} \coloneqq \bigl(\symbvec{p}_{\mathrm{O}}, \BoldtraceDiv[\mathrm{B}](\Phi), \BoldtraceDiv[\mathrm{F}](\Phi)\bigr)\). It holds
  \begin{equation*}
    \llangle \symbvec{p}, \symbvec{v} \rrangle = 0 \ \forall \bm{v} \in \mathbb{X}_{\mathrm{D}}(\Sigma)
    \iff
    \begin{cases}
      \Phi \in H_{\mathrm{loc}}(\Div,\overline{\Omega_{\mathrm{O}}^c}), \\
      \BoldtraceDivC[\mathrm{O}](\Phi) = \symbvec{p}_{\mathrm{O}}.
    \end{cases}
  \end{equation*}
\end{proposition}
\begin{proof}
  (\(\Longleftarrow\)) We simply need to establish that \(\symbvec{p} \in \mathbb{X}_{\mathrm{N}}(\Sigma)\) to conclude.
  On one hand, we extend \(\Phi\) by \(0\) in \(\Omega_{\mathrm{O}}\), so that \(\Phi \in L^2_{\mathrm{loc}}(\mathbb{R}^d)^d\).
  On the other hand, there exists \(\Psi \in H_{\mathrm{loc}}(\Div,\overline{\Omega_{\mathrm{O}}})\) such that \(\BoldtraceDiv[\mathrm{O}](\Psi) = \symbvec{p}_{\mathrm{O}}\), because \(\BoldtraceDiv[\mathrm{O}]\) is onto. 
  Extending \(\Psi\) by \(0\) in \(\Omega_{\mathrm{O}}^c\), it comes \(\Phi + \Psi \in H_{\mathrm{loc}}(\Div,\mathbb{R}^d)\).

  Next, we consider \(\chi \in \mathcal{C}^{\infty}_C(\mathbb{R}^d)\) such that \(\supp(\chi) \supset \Sigma\). Note that such a choice is possible because \(\Sigma\) is bounded.
  Therefore, \(V = \chi(\Phi + \Psi)\) belongs to \(H(\Div,\mathbb{R}^d)\). Eventually, \(\symbvec{p} = \bigl(\BoldtraceDiv[\mathrm{O}](V), \BoldtraceDiv[\mathrm{B}](V), \BoldtraceDiv[\mathrm{F}](V)\bigr)\) belongs to \(\mathbb{X}_{\mathrm{N}}(\Sigma)\).

  (\(\Longrightarrow\))
  First, we show that \(\Phi \in H_{\mathrm{loc}}(\Div,\overline{\Omega_{\mathrm{O}}^c})\). 
  Let \(\varphi \in H^1_{\mathrm{comp}}(\Omega_{\mathrm{O}}^c)\), where \(H^1_{\mathrm{comp}}(\Omega_{\mathrm{O}}^c)\) denotes the elements of \(H^1(\Omega_{\mathrm{O}}^c)\) with bounded support.
  The vector
  \(
  \symbvec{v} \coloneqq 
  \bigl(
    \BoldtraceDirC[\mathrm{O}](\varphi), \BoldtraceDir[\mathrm{B}](\varphi), \BoldtraceDir[\mathrm{F}](\varphi)
  \bigr) 
  = 
  \bigl(
    \symbvec{0}_{\mathrm{O}}, \BoldtraceDir[\mathrm{B}](\varphi), \BoldtraceDir[\mathrm{F}](\varphi)
  \bigr)
  \) belongs to \(\mathbb{X}_{\mathrm{D}}(\Sigma)\), the proof relying on computations similar to those used in the first part (\(\Longleftarrow\))  of the proof.
  The relation \(0 = \llangle \symbvec{v}, \symbvec{p} \rrangle\) becomes
  \begin{equation}\label{eq:specific_polar}
    0 = \langle \BoldtraceDir[\mathrm{F}](\varphi), \symbvec{p}_{\mathrm{F}} \rangle
    + \langle \BoldtraceDir[\mathrm{B}](\varphi), \symbvec{p}_{\mathrm{B}} \rangle
    =
    \langle \BoldtraceDir[\mathrm{F}](\varphi),  \BoldtraceDiv[\mathrm{F}](\Phi) \rangle
    + \langle \BoldtraceDir[\mathrm{B}](\varphi),  \BoldtraceDiv[\mathrm{B}](\Phi) \rangle.
  \end{equation}
  Next, define \(q^* \in L^2(\Omega_{\mathrm{O}}^c)\) by \(q^*|_{\Omega_{\alpha}} = \Div(\Phi |_{\Omega_{\alpha}})\) for \(\alpha \in \setMI_{\mathrm{F}} \cup \setMI_{\mathrm{B}}\).
  Using Equation~\eqref{eq:specific_polar} and Green's first identity, we establish that \(p^* \in H_{\mathrm{loc}}(\Div,\overline{\Omega_{\mathrm{O}}^c})\):
  \begin{equation*}
    \begin{split}
    -\int_{\Omega_{\mathrm{O}}^c} \nabla \varphi \cdot \Phi 
    &= 
    - \sum_{\alpha \in \setMI_{\mathrm{F}}} \int_{\Omega_{\alpha}} \nabla \varphi |_{\Omega_{\alpha}} \cdot \Phi |_{\Omega_{\alpha}}
    - \sum_{\beta \in \setMI_{\mathrm{B}}} \int_{\Omega_{\beta}} \nabla \varphi |_{\Omega_{\beta}} \cdot \Phi |_{\Omega_{\beta}} \\
    &=
    - \sum_{\alpha \in \setMI_{\mathrm{F}}} \int_{\Omega_{\alpha}} \nabla \varphi |_{\Omega_{\alpha}} \cdot \Phi |_{\Omega_{\alpha}}
    - \sum_{\beta \in \setMI_{\mathrm{B}}} \int_{\Omega_{\beta}} \nabla \varphi |_{\Omega_{\beta}} \cdot \Phi |_{\Omega_{\beta}} \\
    &=
    \sum_{\alpha \in \setMI_{\mathrm{F}}} \int_{\Omega_{\alpha}} \varphi |_{\Omega_{\alpha}} \Div(\Phi |_{\Omega_{\alpha}})
    + \sum_{\beta \in \setMI_{\mathrm{B}}} \int_{\Omega_{\beta}} \varphi |_{\Omega_{\beta}} \Div(\Phi |_{\Omega_{\beta}})
    - \langle \BoldtraceDir[\mathrm{F}](\varphi),  \BoldtraceDiv[\mathrm{F}](\Phi) \rangle
    - \langle \BoldtraceDir[\mathrm{B}](\varphi), \ \BoldtraceDiv[\mathrm{B}](\Phi) \rangle \\
    &=
    \sum_{\alpha \in \setMI_{\mathrm{F}}} \int_{\Omega_{\alpha}} \varphi |_{\Omega_{\alpha}} \Div(\Phi |_{\Omega_{\alpha}})
    + \sum_{\beta \in \setMI_{\mathrm{B}}} \int_{\Omega_{\beta}} \varphi |_{\Omega_{\beta}} \Div(\Phi |_{\Omega_{\beta}})
    = \int_{\Omega_{\mathrm{O}}^c} \varphi\; q^*.
    \end{split}
  \end{equation*}

  To end the proof, it remains to show that \(\BoldtraceDivC[\mathrm{O}](\Phi) = \symbvec{p}_{\mathrm{O}} \).
  Let \(v^* \in H^1(\mathbb{R}^d)\). By definition \( \bigl(\BoldtraceDir[\mathrm{O}](v^*), \BoldtraceDir[\mathrm{B}](v^*), \BoldtraceDir[\mathrm{F}](v^*)\bigr) \in \mathbb{X}_{\mathrm{D}}(\Sigma)\). The third hypothesis of the proposition leads to
  \begin{equation*}
    \langle \BoldtraceDir[\mathrm{O}](v^*), \symbvec{p}_{\mathrm{O}} \rangle
    + \langle \BoldtraceDir[\mathrm{B}](v^*), \BoldtraceDiv[\mathrm{B}](\Phi) \rangle
    + \langle \BoldtraceDir[\mathrm{F}](v^*), \BoldtraceDiv[\mathrm{F}](\Phi) \rangle
    = 0.
  \end{equation*}  
  Additionally, in the first part (\(\Longleftarrow\)) of the proof, we have shown that \(\bigl(\BoldtraceDivC[\mathrm{O}](\Phi), \BoldtraceDiv[\mathrm{B}](\Phi), \BoldtraceDiv[\mathrm{F}](\Phi)\bigr) \in \mathbb{X}_{\mathrm{N}}(\Sigma)\). So, we derive
  \[
  \langle \BoldtraceDir[\mathrm{O}](v^*), \BoldtraceDivC[\mathrm{O}](\Phi) \rangle
    + \langle \BoldtraceDir[\mathrm{B}](v^*), \BoldtraceDiv[\mathrm{B}](\Phi) \rangle
    + \langle \BoldtraceDir[\mathrm{F}](v^*), \BoldtraceDiv[\mathrm{F}](\Phi) \rangle = 0.
  \]
  Therefore, \( \langle \BoldtraceDir[\mathrm{O}](v^*),\BoldtraceDivC[\mathrm{O}](p^*) - \symbvec{p}_{\mathrm{O}} \rangle = 0\) for all \( v^* \in H^1(\mathbb{R}^d)\). 
  Recalling that \(\Gamma_{\mathrm{O}}\) is bounded, we know that for any given \(\symbvec{v}_{\mathrm{O}} \in \mathbb{H}_{\mathrm{D}}(\Gamma_{\mathrm{O}})\), there exists \(v^* \in H^1(\mathbb{R}^d)\) such that \( \BoldtraceDir[\mathrm{O}](v^*) = \symbvec{v}_{\mathrm{O}}\).
  Eventually, \(\BoldtraceDivC[\mathrm{O}](p^*) = \symbvec{p}_{\mathrm{O}} \).

\end{proof}
\begin{proposition}\label{prop:neumann_annihilator}
  Let \(w \in L^2_{\mathrm{loc}}(\overline{\Omega_{\mathrm{O}}^c})\) such that \(w |_{\Omega_{\alpha}} \in H^1(\Omega_{\alpha})\) and \(w |_{\Omega_{\beta}} \in H^1_{\mathrm{loc}}(\overline{\Omega_{\beta}})\), for all \(\alpha \in \setMI_{\mathrm{F}}\) and \(\beta \in \setMI_{\mathrm{B}}\).
  Let \(\symbvec{u}_{\mathrm{O}} \in \mathbb{H}_{\mathrm{D}}(\Gamma_{\mathrm{O}})\), and define \(\symbvec{u} \coloneqq \bigl(\symbvec{u}_{\mathrm{O}}, \BoldtraceDir[\mathrm{B}](w), \BoldtraceDir[\mathrm{F}](w)\bigr)\). It holds
  \begin{equation*}
    \llangle \symbvec{q}, \symbvec{u} \rrangle = 0 \ \forall \symbvec{q} \in \mathbb{X}_{\mathrm{N}}(\Sigma)
    \iff
    \begin{cases}
      w \in H^1_{\mathrm{loc}}(\overline{\Omega_{\mathrm{O}}^c}), \\
      \BoldtraceDirC[\mathrm{O}](w) = \symbvec{u}_{\mathrm{O}}.
    \end{cases}
  \end{equation*}
\end{proposition}

\section{The multi-domain variational formulation}\label{sec:varf}

We introduce the function space
\begin{equation*}
  \begin{split}
  \mathbb{X}(\Omega_{\mathrm{O}}^c) \coloneqq 
  \Bigl\{
  &\bigl(\left(\phi_{\delta}, p_{\delta}\right)_{\setMI_{\mathrm{O}}}, \left(\phi_{\beta}, p_{\beta}\right)_{\setMI_{\mathrm{B}}}, (u_\alpha)_{\setMI_{\mathrm{F}}}\bigr) \in \mathbb{H}(\Gamma_{\mathrm{O}}) \times \mathbb{H}(\Gamma_{\mathrm{B}}) \times \mathbb{H}^1(\Omega_{\mathrm{F}}) \mid \\
  &\exists u^* \in H^1(\mathbb{R}^d) \text{ such that } \BoldtraceDir[\mathrm{O}](u^*) = \symbvec{\phi}_{\mathrm{O}},\ \BoldtraceDir[\mathrm{B}](u^*) = \symbvec{\phi}_{\mathrm{B}}, \text{ and } u^* |_{\Omega_{\alpha}} = u_{\alpha}, \forall \alpha \in \setMI_{\mathrm{F}}
  \Bigr\}
  \end{split}
\end{equation*}
for the unknown of the sought multi-domain variational formulation. This space is nothing but the space of tuples of functions that satisfy \emph{Dirichlet} transmission conditions on the skeleton of the domain decomposition. 
It could have been denoted \(\mathbb{X}(\mathbb{R}^d)\), but we have chosen to use the notation \(\mathbb{X}(\Omega_{\mathrm{O}}^c)\) to emphasize that some boundary conditions may be involved in Problem~\eqref{pbm:initialPb}.

By defining the Dirichlet trace operator \(\B\) as 
\begin{equation*}
  \begin{split}
      \B \colon & \mathbb{H}(\Gamma_{\mathrm{O}}) \times \mathbb{H}(\Gamma_{\mathrm{B}}) \times \mathbb{H}^1(\Omega_{\mathrm{F}}) \longrightarrow \mathbb{H}_{\mathrm{D}}(\Sigma), \\ 
      & \bigl(\left(\phi_{\delta}, p_{\delta}\right)_{\setMI_{\mathrm{O}}}, \left(\phi_{\beta}, p_{\beta}\right)_{\setMI_{\mathrm{B}}}, (u_\alpha)_{\setMI_{\mathrm{F}}}\bigr) 
  \longmapsto
  \bigl( \symbvec{\phi}_{\mathrm{O}}, \symbvec{\phi}_{\mathrm{B}}, \BoldtraceDir[\mathrm{F}](\symbvec{u}_{\mathrm{F}}) \bigr),  
  \end{split}
\end{equation*}
we can observe that \(\mathbb{X}_{\mathrm{D}}(\Sigma) = \B\left(\mathbb{X}(\Omega_{\mathrm{O}}^c)\right)\). Once again, \(\mathbb{X}_{\mathrm{D}}(\Sigma)\) can be interpreted as the space of Dirichlet traces of functions that match on the skeleton \(\Sigma\). We can now state the sought variational formulation of Problem~\eqref{pbm:multi_domain_config} (equivalently, of Problem~\eqref{pbm:initialPb}).

\begin{proposition}\label{prop:multi_domain_variational_direct}
  Let \(u \in H^1_{\mathrm{loc}}(\Delta, \overline{\Omega_{\mathrm{O}}^c})\) be a solution to Problem~\eqref{pbm:multi_domain_config}. Then, \(
  \bigl(\symbvec{u}_{\mathrm{O}}, \symbvec{u}_{\mathrm{B}}, \symbvec{u}_{\mathrm{F}}\bigr) 
  \coloneqq 
  \bigl(
  \symbvec{\gamma}_c^{\mathrm{O}}(u), \symbvec{\gamma}^{\mathrm{B}}(u), \left(u |_{\Omega_{\alpha}}\right)_{\setMI_{\mathrm{F}}}
  \bigr)
  \)
  belongs to \(\mathbb{X}(\Omega_{\mathrm{O}}^c)\), and 
  is solution to
  \begin{equation}\label{pbm:single_domain_formulation}
    \begin{split}
      &\langle \symbvec{\A}_{\Gamma_{\mathrm{O}}} \symbvec{u}_{\mathrm{O}}, \symbvec{v}_{\mathrm{O}} \rangle
      + \langle \symbvec{\A}_{\Gamma_{\mathrm{B}}} \symbvec{u}_{\mathrm{B}}, \symbvec{v}_{\mathrm{B}} \rangle
      + \langle \symbvec{\A}_{\Omega_{\mathrm{F}}} \symbvec{u}_{\mathrm{F}}, \symbvec{v}_{\mathrm{F}} \rangle
      = 
      \langle \symbvec{f}_{\mathrm{F}}, \symbvec{v}_{\mathrm{F}} \rangle + 
      \langle \symbvec{g}_{D}, \symbvec{q}_{\mathrm{O}} \rangle, \\
      &\forall \left(\symbvec{v}_{\mathrm{O}}, \symbvec{v}_{\mathrm{B}}, \symbvec{v}_{\mathrm{F}}\right) 
      \coloneqq 
      \bigl(
      \left(\psi_{\delta}, q_{\delta}\right)_{\setMI_{\mathrm{O}}}, \left(\psi_{\beta}, q_{\beta}\right)_{\setMI_{\mathrm{B}}}, \left(v_{\alpha}\right)_{\setMI_{\mathrm{F}}}
      \bigr)
      \in \mathbb{X}(\Omega_{\mathrm{O}}^c)
    \end{split}
  \end{equation}
  with \(\symbvec{\A}_{\Gamma_{\bullet}} \coloneqq \diag(\A_{\Gamma_{\bullet,m}})_m\), \, \(\A_{\Gamma_{\bullet,m}} \colon \mathbb{H}(\Gamma_{\bullet,m}) \to {\mathbb{H}(\Gamma_{\bullet,m})}^* \),
  \begin{align*}
    &\A_{\Gamma_{\mathrm{O},j}} \coloneqq
    \begin{bmatrix}
      0 & \Id \\
      \Id & 0
    \end{bmatrix} \text{ and }
    \A_{\Gamma_{\mathrm{B},k}} \coloneqq 
    \begin{bmatrix}
      \W_{\mathrm{B},k} & \Id/2 + \Ktilde_{\mathrm{B},k} \\
      \Id/2 - \K_{\mathrm{B},k} & -\V_{\mathrm{B},k}
    \end{bmatrix},\\
    &\langle \symbvec{\A}_{\Omega_{\mathrm{F}}} \symbvec{u}_{\mathrm{F}}, \symbvec{v}_{\mathrm{F}} \rangle \coloneqq \sum_{\alpha \in \setMI_{\mathrm{F}}}
    \int_{\Omega_{\alpha}} \nabla u|_{\Omega_{\alpha}} \cdot \nabla v_{\alpha} - \kappa |_{\Omega_{\alpha}}^2 u|_{\Omega_{\alpha}} v_{\alpha}.
  \end{align*}
  The definition of the operators \(\A_{\Gamma_{\mathrm{B},k}}\) is the one arising from the \emph{Costabel FEM-BEM coupling}, see~\cite{Costabel1987SMC} and~\cite[Example~3.1]{BoisneaultBonazzoliEtAl2026DFB}.
\end{proposition}

\begin{proof}
  Since \(u \in H^1_{\mathrm{loc}}(\Delta, \overline{\Omega_{\mathrm{O}}^c})\), \(u|_{\Omega_{\alpha}} \in H^1(\Omega_{\alpha})\) for all \(\alpha \in \setMI_{\mathrm{F}}\). For \(\symbvec{v}_{\mathrm{F}} \in \mathbb{H}^1(\Omega_{\mathrm{F}})\), using Green's first identity we derive
  \begin{equation}\label{eq:multi_domain_green_formula}
    \begin{split}
      \langle \symbvec{f}_{\mathrm{F}}, \symbvec{v}_{\mathrm{F}} \rangle
      &=
      \sum_{\alpha \in \setMI_{\mathrm{F}}}\int_{\Omega_{\alpha}} \bigl(-\Delta u|_{\Omega_{\alpha}} - \kappa|_{\Omega_{\alpha}}^2 u|_{\Omega_{\alpha}}\bigr) v_{\alpha}\\
      &=
      \sum_{\alpha \in \setMI_{\mathrm{F}}} \int_{\Omega_{\alpha}} (\nabla u|_{\Omega_{\alpha}} \cdot \nabla v_{\alpha} - \kappa |_{\Omega_{\alpha}}^2 u|_{\Omega_{\alpha}} v_{\alpha})
      - \langle \BoldtraceNeu[\mathrm{F}](u), \BoldtraceDir[\mathrm{F}](\symbvec{v}_{\mathrm{F}}) \rangle\\
      &=
      \langle \symbvec{\A}_{\Omega_{\mathrm{F}}} \symbvec{u}_{\mathrm{F}}, \symbvec{v}_{\mathrm{F}} \rangle
      - \langle \BoldtraceNeu[\mathrm{F}](u), \BoldtraceDir[\mathrm{F}](\symbvec{v}_{\mathrm{F}}) \rangle.      
    \end{split}     
  \end{equation}
  Afterwards, \(u \in H^1_{\mathrm{loc}}(\overline{\Omega_{\mathrm{O}}^c})\) implies that \(
  \bigl(
  \symbvec{\gamma}_c^{\mathrm{O}}(u), \symbvec{\gamma}^{\mathrm{B}}(u), \left(u |_{\Omega_{\alpha}}\right)_{\setMI_{\mathrm{F}}}
  \bigr)
  \in \mathbb{X}(\Omega_{\mathrm{O}}^c)
  \)---which is the space in which we seek the solution to the variational formulation. Additionally, since \(u \in H^1_{\mathrm{loc}}(\Delta,\overline{\Omega_{\mathrm{O}}^c})\), we derive from Proposition~\ref{prop:neumann_annihilator} that \(
  \bigl(
  \BoldtraceNeuC[\mathrm{O}](u), \BoldtraceNeu[\mathrm{B}](u), \BoldtraceNeu[\mathrm{F}](u)
  \bigr)
  \in \mathbb{X}_{\mathrm{N}}(\Sigma)
  \).
  Thus, Proposition~\ref{prop:annihilators} implies that
  \begin{equation}\label{eq:prop4_neumann_transm}
    -\langle \BoldtraceNeu[\mathrm{F}](u), \symbvec{\psi}_{\mathrm{F}} \rangle
    =
    \langle \BoldtraceNeu[\mathrm{B}](u), \symbvec{\psi}_{\mathrm{B}} \rangle
    + \langle \BoldtraceNeuC[\mathrm{O}](u), \symbvec{\psi}_{\mathrm{O}} \rangle \quad
    \text{for all } \left(\symbvec{\psi}_{\mathrm{O}}, \symbvec{\psi}_{\mathrm{B}}, \symbvec{\psi}_{\mathrm{F}}\right) \in \mathbb{X}_{\mathrm{D}}(\Sigma).
  \end{equation}
  Let 
  \(
  \bigl(
   \left(\psi_{\delta}, q_{\delta}\right)_{\setMI_{\mathrm{O}}}, \left(\psi_{\beta}, q_{\beta}\right)_{\setMI_{\mathrm{B}}}, (v_\alpha)_{\setMI_{\mathrm{F}}}
  \bigr)  
  \in \mathbb{X}(\Omega_{\mathrm{O}}^c)
  \), 
  we apply~\eqref{eq:prop4_neumann_transm} to the Dirichlet trace of the previous tuple, namely
  \(
  \bigl(
  \symbvec{\psi}_{\mathrm{O}}, \symbvec{\psi}_{\mathrm{B}}, \BoldtraceDir[\mathrm{F}](\symbvec{v}_{\mathrm{F}})
  \bigr)
  \) (given by the operator \(\B\)), 
  and inject it into \cref{eq:multi_domain_green_formula} to obtain
  \begin{equation}\label{eq:multi_domain_green_formula_injected_neumann}
    \langle \symbvec{\A}_{\Omega_{\mathrm{F}}} \left(u |_{\Omega_{\alpha}}\right)_{\setMI_{\mathrm{F}}}, \symbvec{v}_{\mathrm{F}} \rangle
    + \langle \BoldtraceNeu[\mathrm{B}](u), \symbvec{\psi}_{\mathrm{B}} \rangle
    + \langle \BoldtraceNeuC[\mathrm{O}](u), \symbvec{\psi}_{\mathrm{O}} \rangle
    = \langle \symbvec{f}_{\mathrm{F}}, \symbvec{v}_{\mathrm{F}} \rangle.    
  \end{equation}
  Moreover, the boundary conditions on the \(\Gamma_{\mathrm{O},j}\) are weakly imposed with the equation 
  \begin{equation}\label{eq:multi_domain_bc}
    \langle \BoldtraceDirC[\mathrm{O}](u), \symbvec{q}_{\mathrm{O}} \rangle
    = 
    \langle \symbvec{g}_{D}, \symbvec{q}_{\mathrm{O}} \rangle \quad \text{for all } \symbvec{q}_{\mathrm{O}} \in \mathbb{H}_{\mathrm{N}}(\Gamma_{\mathrm{O}}).       
  \end{equation}
  
  Now, we are left to deal with the contributions from the homogeneous domains \(\Omega_{\mathrm{B},k}\). We do so using BIEs. Indeed, since \(-\Delta u - \kappa_k^2 u = 0\) in any \(\Omega_{\mathrm{B},k}\), we can use the Calderón equations~\eqref{eq:calderon_1} and~\eqref{eq:calderon_2} for all \(k\).
  Injecting the second Calderón equation~\eqref{eq:calderon_2} into \cref{eq:multi_domain_green_formula_injected_neumann}, writing variationally the first Calderón equation~\eqref{eq:calderon_1}, and considering \cref{eq:multi_domain_bc} to impose the boundary conditions for the impenetrable obstacles, we derive the following system (where the capital bold letters are block diagonal operators defined subdomainwise):
  \begin{equation*}
    \begin{split}
      &\begin{cases}
        \begin{aligned}
          \langle \symbvec{\A}_{\Omega_{\mathrm{F}}} \left(u |_{\Omega_{\alpha}}\right)_{\setMI_{\mathrm{F}}}, \symbvec{v}_{\mathrm{F}} \rangle
          + \langle \symbvec{\W}_{\mathrm{B}} \BoldtraceDir[\mathrm{B}](u), \symbvec{\psi}_{\mathrm{B}} \rangle    
          + \langle \left(\symbvec{\Id}/2 + \symbvec{\Ktilde}_{\mathrm{B}}\right) \BoldtraceNeu[\mathrm{B}](u), \symbvec{\psi}_{\mathrm{B}} \rangle 
          + \langle \BoldtraceNeuC[\mathrm{O}](u), \symbvec{\psi}_{\mathrm{O}} \rangle 
          &= \langle \symbvec{f}_{\mathrm{F}}, \symbvec{v}_{\mathrm{F}} \rangle, \\
          \langle \left(\symbvec{\Id}/2 - \symbvec{\K}_{\mathrm{B}}\right) \BoldtraceDir[\mathrm{B}](u), \symbvec{q}_{\mathrm{B}} \rangle
          - \langle \symbvec{\V}_{\mathrm{B}} \BoldtraceNeu[\mathrm{B}](u) , \symbvec{q}_{\mathrm{B}} \rangle
          &= 0, \\
          \langle \BoldtraceDirC[\mathrm{O}](u), \symbvec{q}_{\mathrm{O}} \rangle
          &= \langle \symbvec{g}_{D}, \symbvec{q}_{\mathrm{O}} \rangle,
        \end{aligned}
      \end{cases}\\
      &\forall \bigl( \left(\psi_{\delta}, q_{\delta}\right)_{\setMI_{\mathrm{O}}}, \left(\psi_{\beta}, q_{\beta}\right)_{\setMI_{\mathrm{B}}}, (v_\alpha)_{\setMI_{\mathrm{F}}} \bigr) \in \mathbb{X}(\Omega_{\mathrm{O}}^c).
    \end{split}
  \end{equation*}
  Finally, denoting 
  \(
  \bigl( 
    \left(\phi_{\delta}, p_{\delta}\right)_{\setMI_{\mathrm{O}}}, \left(\phi_{\beta}, p_{\beta}\right)_{\setMI_{\mathrm{B}}}, (u_\alpha)_{\setMI_{\mathrm{F}}} 
  \bigr)
  \coloneqq
  \bigl(
  \symbvec{\gamma}_c^{\mathrm{O}}(u), \symbvec{\gamma}^{\mathrm{B}}(u), \left(u |_{\Omega_{\alpha}}\right)_{\setMI_{\mathrm{F}}}
  \bigr)
  \) 
  the system writes
  \begin{equation*}
    \begin{split}
      &\text{Find } \bigl( \left(\phi_{\delta}, p_{\delta}\right)_{\setMI_{\mathrm{O}}}, \left(\phi_{\beta}, p_{\beta}\right)_{\setMI_{\mathrm{B}}}, (u_\alpha)_{\setMI_{\mathrm{F}}} \bigr)  \text{ in } \mathbb{X}(\Omega_{\mathrm{O}}^c) \text{ such that} \\
      &\begin{cases}
        \begin{aligned}
          \langle \symbvec{\A}_{\Omega_{\mathrm{F}}} \symbvec{u}_{\mathrm{F}}, \symbvec{v}_{\mathrm{F}} \rangle
          + \langle \symbvec{\W}_{\mathrm{B}} \symbvec{\phi}_{\mathrm{B}}, \symbvec{\psi}_{\mathrm{B}} \rangle      
          + \langle \left(\symbvec{\Id}/2 + \symbvec{\Ktilde}_{\mathrm{B}}\right) \symbvec{p}_{\mathrm{B}}, \symbvec{\psi}_{\mathrm{B}} \rangle      
          + \langle \symbvec{p}_{\mathrm{O}}, \symbvec{\psi}_{\mathrm{O}} \rangle 
          &= \langle \symbvec{f}_{\mathrm{F}}, \symbvec{v}_{\mathrm{F}} \rangle, \\
          \langle \left(\symbvec{\Id}/2 - \symbvec{\K}_{\mathrm{B}}\right) \symbvec{\phi}_{\mathrm{B}}, \symbvec{q}_{\mathrm{B}} \rangle
          - \langle \symbvec{\V}_{\mathrm{B}} \symbvec{p}_{\mathrm{B}} , \symbvec{q}_{\mathrm{B}} \rangle
          &= 0, \\
          \langle \symbvec{\phi}_{\mathrm{O}}, \symbvec{q}_{\mathrm{O}} \rangle
          &= \langle \symbvec{g}_{D}, \symbvec{q}_{\mathrm{O}} \rangle,
        \end{aligned}
      \end{cases}\\
      &\forall \ \bigl( \left(\psi_{\delta}, q_{\delta}\right)_{\setMI_{\mathrm{O}}}, \left(\psi_{\beta}, q_{\beta}\right)_{\setMI_{\mathrm{B}}}, (v_\alpha)_{\setMI_{\mathrm{F}}} \bigr) \in \mathbb{X}(\Omega_{\mathrm{O}}^c).  
    \end{split}
  \end{equation*}
  Eventually, summing the three lines gives the sought formulation~\eqref{pbm:single_domain_formulation}.
\end{proof}

We now prove that, given a solution to the variational formulation~\eqref{pbm:single_domain_formulation}, we can reconstruct a solution to the boundary value problem~\eqref{pbm:multi_domain_config}.  
\begin{proposition}\label{proposition:multi_domain_variational_reciprocal}
  Let \(
  \bigl(\symbvec{u}_{\mathrm{O}}, \symbvec{u}_{\mathrm{B}}, \symbvec{u}_{\mathrm{F}}\bigr) 
  \coloneqq 
  \bigl(
  \left(\phi_{\delta}, p_{\delta}\right)_{\setMI_{\mathrm{O}}}, \left(\phi_{\beta}, p_{\beta}\right)_{\setMI_{\mathrm{B}}}, \left(u_{\alpha}\right)_{\setMI_{\mathrm{F}}}
  \bigr)
  \in \mathbb{X}(\Omega_{\mathrm{O}}^c)\) be a solution to
  \begin{equation}
    \label{pbm:single_domain_formulation_multi_domain_bis}    
    \begin{split}
      &\langle \symbvec{\A}_{\Gamma_{\mathrm{O}}} \symbvec{u}_{\mathrm{O}}, \symbvec{v}_{\mathrm{O}} \rangle
      + \langle \symbvec{\A}_{\Gamma_{\mathrm{B}}} \symbvec{u}_{\mathrm{B}}, \symbvec{v}_{\mathrm{B}} \rangle
      + \langle \symbvec{\A}_{\Omega_{\mathrm{F}}} \symbvec{u}_{\mathrm{F}}, \symbvec{v}_{\mathrm{F}} \rangle
      = 
      \langle \symbvec{f}_{\mathrm{F}}, \symbvec{v}_{\mathrm{F}} \rangle + 
      \langle \symbvec{g}_{D}, \symbvec{q}_{\mathrm{O}} \rangle, \\
      &\forall \left(\symbvec{v}_{\mathrm{O}}, \symbvec{v}_{\mathrm{B}}, \symbvec{v}_{\mathrm{F}}\right) 
      \coloneqq 
      \bigl( 
       \left(\psi_{\delta}, q_{\delta}\right)_{\setMI_{\mathrm{O}}}, \left(\psi_{\beta}, q_{\beta}\right)_{\setMI_{\mathrm{B}}}, \left(v_{\alpha}\right)_{\setMI_{\mathrm{F}}} 
      \bigr)
      \in \mathbb{X}(\Omega_{\mathrm{O}}^c),
    \end{split}
  \end{equation}
  with the same definitions for the operators as in~Proposition~\ref{prop:multi_domain_variational_direct}.
  Then,
  \begin{equation*}
    w(x) \coloneqq
    \begin{cases}
      \GL_{\beta}\left(\phi_{\beta}, p_{\beta}\right)(x), & \text{ in } \Omega_{\beta}, \quad \beta \in \setMI_{\mathrm{B}} \\
      u_{\alpha}(x), & \text{ in } \Omega_{\alpha} , \quad \alpha \in \setMI_{\mathrm{F}}
    \end{cases}
  \end{equation*}
  belongs to \(H^1_{\mathrm{loc}}(\Delta, \overline{\Omega_{\mathrm{O}}^c})\) and solves Problem~\eqref{pbm:multi_domain_config}.
\end{proposition}

\begin{proof}
  \ding{172} Let \(\left(\symbvec{v}_{\mathrm{O}}, \symbvec{v}_{\mathrm{B}}, \symbvec{v}_{\mathrm{F}}\right)  = \left(\symbvec{0}_{\mathrm{O}}, \symbvec{0}_{\mathrm{B}}, \left(v_{\alpha}\right)_{\setMI_{\mathrm{F}}}\right)\) with \(v_{\alpha} \in \mathcal{C}^{\infty}_{c}(\Omega_{\alpha}) \,\text{ for all } \alpha \in \setMI_{\mathrm{F}}\). By construction, \(\left(\symbvec{v}_{\mathrm{O}}, \symbvec{v}_{\mathrm{B}}, \symbvec{v}_{\mathrm{F}}\right) \in \mathbb{X}(\Omega_{\mathrm{O}}^c)\). Thus, \cref{pbm:single_domain_formulation_multi_domain_bis} becomes
  \begin{equation*}    
    \begin{aligned}    
      \langle \symbvec{\A}_{\Omega_{\mathrm{F}}} \symbvec{u}_{\mathrm{F}}, \symbvec{v}_{\mathrm{F}} \rangle = \langle \symbvec{f}_{\mathrm{F}}, \symbvec{v}_{\mathrm{F}} \rangle
      &\iff 
      \sum_{\alpha \in \setMI_{\mathrm{F}}} \int_{\Omega_{\alpha}} (\nabla u_{\alpha} \cdot \nabla v_{\alpha} - \kappa |_{\Omega_{\alpha}}^2 u_{\alpha} v_{\alpha}) = \sum_{\alpha \in \setMI_{\mathrm{F}}} \int_{\Omega_{\alpha}} f|_{\Omega_{\alpha}}\, v_{\alpha}  \\
      &\iff 
      \sum_{\alpha \in \setMI_{\mathrm{F}}} \langle -\Delta u_{\alpha} - \kappa |_{\Omega_{\alpha}}^2 u_{\alpha}, v_{\alpha} \rangle = \sum_{\alpha \in \setMI_{\mathrm{F}}} \langle f|_{\Omega_{\alpha}}, v_{\alpha} \rangle,        
    \end{aligned} 
  \end{equation*}
  where in the last equivalence we applied Green's first identity. Thus, \(-\Delta u_{\alpha} - \kappa|_{\Omega_{\alpha}}^2 u_{\alpha} = f|_{\Omega_{\alpha}}\) in \(\Omega_{\alpha}\) in the sense of distributions. We deduce that \(u_{\alpha} \in H^1(\Delta, \Omega_{\alpha})\) because \(f|_{\Omega_{\alpha}} \in L^2(\Omega_{\alpha})\) and \(u_{\alpha} \in H^1(\Omega_{\alpha})\)
  
  \ding{173} We get back to general test functions \(\left(\symbvec{v}_{\mathrm{O}}, \symbvec{v}_{\mathrm{B}}, \symbvec{v}_{\mathrm{F}}\right) \in \mathbb{X}(\Omega_{\mathrm{O}}^c)\) and apply Green's first identity in the left-hand side of \cref{pbm:single_domain_formulation_multi_domain_bis}, using that \(-\Delta u_{\alpha} - \kappa|_{\Omega_{\alpha}}^2 u_{\alpha} = f|_{\Omega_{\alpha}}\) in \(\Omega_{\alpha}\):
  \begin{equation*}
    \begin{split}
    &\langle \symbvec{\A}_{\Gamma_{\mathrm{O}}} \symbvec{u}_{\mathrm{O}}, \symbvec{v}_{\mathrm{O}} \rangle
    + \langle \symbvec{\A}_{\Gamma_{\mathrm{B}}} \symbvec{u}_{\mathrm{B}}, \symbvec{v}_{\mathrm{B}} \rangle
    + \langle \symbvec{\A}_{\Omega_{\mathrm{F}}} \symbvec{u}_{\mathrm{F}}, \symbvec{v}_{\mathrm{F}} \rangle
    = \langle \symbvec{f}_{\mathrm{F}}, \symbvec{v}_{\mathrm{F}} \rangle +
    \langle \symbvec{g}_{D}, \symbvec{q}_{\mathrm{O}} \rangle \\
    \iff
    &\langle \symbvec{\A}_{\Gamma_{\mathrm{O}}} \symbvec{u}_{\mathrm{O}}, \symbvec{v}_{\mathrm{O}} \rangle
    + \langle \symbvec{\A}_{\Gamma_{\mathrm{B}}} \symbvec{u}_{\mathrm{B}}, \symbvec{v}_{\mathrm{B}} \rangle
    + \langle \BoldtraceNeu[\mathrm{F}](\symbvec{u}_{\mathrm{F}}), \BoldtraceDir[\mathrm{F}](\symbvec{v}_{\mathrm{F}}) \rangle
    = 
    \langle \symbvec{g}_{D}, \symbvec{q}_{\mathrm{O}} \rangle \\
    \iff
    &\langle \symbvec{\A}_{\Gamma_{\mathrm{O}}} \symbvec{u}_{\mathrm{O}}, \symbvec{v}_{\mathrm{O}} \rangle
    + \langle \symbvec{\A}_{\Gamma_{\mathrm{B}}} \symbvec{u}_{\mathrm{B}}, \symbvec{v}_{\mathrm{B}} \rangle
    + \langle \BoldtraceNeu[\mathrm{F}](\symbvec{u}_{\mathrm{F}}), \BoldtraceDir[\mathrm{F}](\symbvec{v}_{\mathrm{F}}) \rangle
    = 
    \langle \symbvec{g}_{D}, \symbvec{q}_{\mathrm{O}} \rangle \\
    \iff
    &\langle \symbvec{p}_{\mathrm{O}}, \symbvec{\psi}_{\mathrm{O}} \rangle
    + \langle \symbvec{\phi}_{\mathrm{O}}, \symbvec{q}_{\mathrm{O}} \rangle
    + \langle \symbvec{\A}_{\Gamma_{\mathrm{B}}} \symbvec{u}_{\mathrm{B}}, \symbvec{v}_{\mathrm{B}} \rangle
    + \langle \BoldtraceNeu[\mathrm{F}](\symbvec{u}_{\mathrm{F}}), \BoldtraceDir[\mathrm{F}](\symbvec{v}_{\mathrm{F}}) \rangle
    = 
    \langle \symbvec{g}_{D}, \symbvec{q}_{\mathrm{O}} \rangle
    \end{split}
  \end{equation*}
  
  \ding{174} Let \(\bigl(\symbvec{v}_{\mathrm{O}}, \symbvec{v}_{\mathrm{B}}, \symbvec{v}_{\mathrm{F}}\bigr) = \bigl(\left(0, q_{\delta}\right)_{\setMI_{\mathrm{O}}}, \symbvec{0}_{\mathrm{B}}, \symbvec{0}_{\mathrm{F}} \bigr)\) with  \(\symbvec{q}_{\mathrm{O}} \in \mathbb{H}_{\mathrm{N}}(\Gamma_{\mathrm{O}})\), which naturally belongs to \(\mathbb{X}(\Omega_{\mathrm{O}}^c)\). Then 
  \(
  \langle \symbvec{\phi}_{\mathrm{O}}, \symbvec{q}_{\mathrm{O}} \rangle = \langle \symbvec{g}_{D}, \symbvec{q}_{\mathrm{O}} \rangle\) for all \( \symbvec{q}_{\mathrm{O}} \in \mathbb{H}_{\mathrm{N}}(\Gamma_{\mathrm{O}})
  \),
  which also writes \(\symbvec{\phi}_{\mathrm{O}} = \symbvec{g}_{D}\) on \(\Gamma_{\mathrm{O}}\). 
  Therefore, the last equation of \ding{173} becomes 
  \begin{equation}\label{eq:varf_after_Dir}
    \langle \symbvec{p}_{\mathrm{O}}, \symbvec{\psi}_{\mathrm{O}} \rangle
    + \langle \symbvec{\A}_{\Gamma_{\mathrm{B}}} \symbvec{u}_{\mathrm{B}}, \symbvec{v}_{\mathrm{B}} \rangle
    + \langle \BoldtraceNeu[\mathrm{F}](\symbvec{u}_{\mathrm{F}}), \BoldtraceDir[\mathrm{F}](\symbvec{v}_{\mathrm{F}}) \rangle
    = 
    0 \quad \text{for all } \bigl(\symbvec{v}_{\mathrm{O}}, \symbvec{v}_{\mathrm{B}}, \symbvec{v}_{\mathrm{F}}\bigr) \in \mathbb{X}(\Omega_{\mathrm{O}}^c).
  \end{equation}
  
  \ding{175} For \(\beta \in \setMI_{\mathrm{B}}\), by the properties of \(\GL_{\beta}\) recalled in Section~\ref{sec:bios}, \(-\Delta w |_{\Omega_{\beta}} - \kappa_k^2 w |_{\Omega_{\beta}} = 0\) in \(\Omega_{\beta}\), and
  \begin{equation}\label{eq:calderon_w}
    \begin{cases}
      \traceDir[\beta](w) = \V_{\beta} p_{\beta} + \left(\Id/2 + \K_{\beta} \right) \phi_{\beta}, \\
      \traceNeu[\beta](w) = \left(\Id/2 + \Ktilde_{\beta} \right) p_{\beta} + \W_{\beta}  \phi_{\beta}.
    \end{cases}                
  \end{equation}
  By using the definition of the \(\A_{\Gamma_{\beta}}\) and injecting~\eqref{eq:calderon_w} into~\eqref{eq:varf_after_Dir}, we obtain
  \begin{equation*}
    \begin{split}
    &\langle \symbvec{p}_{\mathrm{O}}, \symbvec{\psi}_{\mathrm{O}} \rangle
    + \langle \symbvec{\A}_{\Gamma_{\mathrm{B}}} \symbvec{u}_{\mathrm{B}}, \symbvec{v}_{\mathrm{B}} \rangle
    + \langle \BoldtraceNeu[\mathrm{F}](\symbvec{u}_{\mathrm{F}}), \BoldtraceDir[\mathrm{F}](\symbvec{v}_{\mathrm{F}}) \rangle
    = 
    0 \\
    \iff 
    &\langle \symbvec{p}_{\mathrm{O}}, \symbvec{\psi}_{\mathrm{O}} \rangle
    + \langle \left(\symbvec{\Id}/2 - \symbvec{\K}_{\mathrm{B}}\right) \symbvec{\phi}_{\mathrm{B}} - \symbvec{\V}_{\mathrm{B}} \symbvec{p}_{\mathrm{B}}, \symbvec{q}_{\mathrm{B}} \rangle
    + \langle \left(\symbvec{\Id}/2 + \symbvec{\Ktilde}_{\mathrm{B}}\right) \symbvec{p}_{\mathrm{B}} + \symbvec{\W}_{\mathrm{B}} \symbvec{\phi}_{\mathrm{B}}, \symbvec{\psi}_{\mathrm{B}} \rangle
    + \langle \BoldtraceNeu[\mathrm{F}](\symbvec{u}_{\mathrm{F}}), \BoldtraceDir[\mathrm{F}](\symbvec{v}_{\mathrm{F}}) \rangle
    = 0 \\
    \iff 
    &\langle \symbvec{p}_{\mathrm{O}}, \symbvec{\psi}_{\mathrm{O}} \rangle
    + \langle  \symbvec{\phi}_{\mathrm{B}} - \BoldtraceDir[\mathrm{B}](w),  \symbvec{q}_{\mathrm{B}} \rangle
    + \langle \BoldtraceNeu[\mathrm{B}](w),  \symbvec{\psi}_{\mathrm{B}} \rangle
    + \langle \BoldtraceNeu[\mathrm{F}](\symbvec{u}_{\mathrm{F}}), \BoldtraceDir[\mathrm{F}](\symbvec{v}_{\mathrm{F}}) \rangle
    = 0.
    \end{split}
  \end{equation*}
  Let \(
  \bigl(
  \left(\psi_{\delta}, q_{\delta}\right)_{\setMI_{\mathrm{O}}}, \left(\psi_{\beta}, q_{\beta}\right)_{\setMI_{\mathrm{B}}}, \symbvec{v}_{\mathrm{F}}
  \bigr) 
  = 
  \bigl(
  \symbvec{0}_{\mathrm{O}}, \left(0, q_{\beta}\right)_{\setMI_{\mathrm{B}}}, \symbvec{0}_{\mathrm{F}} 
  \bigr) \in \mathbb{X}(\Omega_{\mathrm{O}}^c)
  \) for any \(\symbvec{q}_{\mathrm{B}} \in \mathbb{H}_{\mathrm{N}}(\Gamma_{\mathrm{B}})\), we derive that \(\BoldtraceDir[\mathrm{B}](w) = \symbvec{\phi}_{\mathrm{B}}\) on \(\Gamma_{\mathrm{B}}\), leaving us with
  \begin{equation*}
    \langle \symbvec{p}_{\mathrm{O}}, \symbvec{\psi}_{\mathrm{O}} \rangle
    + \langle \BoldtraceNeu[\mathrm{B}](w),  \symbvec{\psi}_{\mathrm{B}} \rangle
    + \langle \BoldtraceNeu[\mathrm{F}](w), \BoldtraceDir[\mathrm{F}](\symbvec{v}_{\mathrm{F}}) \rangle
    = 0 \quad \text{for all }
      \bigl( 
        \left(\psi_{\delta}, q_{\delta}\right)_{\setMI_{\mathrm{O}}}, \left(\psi_{\beta}, q_{\beta}\right)_{\setMI_{\mathrm{B}}}, \symbvec{v}_{\mathrm{F}} 
      \bigr)
      \in \mathbb{X}(\Omega_{\mathrm{O}}^c),
  \end{equation*}
  ---~let us recall that by definition \(w|_{\Omega_{\alpha}} = u_{\alpha}\) for all \(\alpha \in \setMI_{\mathrm{F}}\).
  
  \ding{176}
  Therefore, since \(\mathbb{X}_{\mathrm{D}}(\Sigma) = \B(\mathbb{X}(\Omega_{\mathrm{O}}^c))\), we get that
  \begin{equation}\label{eq:prop5_neumann_transm}
    \langle \symbvec{p}_{\mathrm{O}}, \symbvec{\psi}_{\mathrm{O}} \rangle
    + \langle \BoldtraceNeu[\mathrm{B}](w),  \symbvec{\psi}_{\mathrm{B}} \rangle
    + \langle \BoldtraceNeu[\mathrm{F}](w), \symbvec{\psi}_{\mathrm{F}} \rangle
    = 0 \quad \text{for all } 
    (
    \symbvec{\psi}_{\mathrm{O}}, \symbvec{\psi}_{\mathrm{B}}, \symbvec{\psi}_{\mathrm{F}}
    )
    \in \mathbb{X}_{\mathrm{D}}(\Sigma) 
  \end{equation}
  holds true. Then, according to Proposition~\ref{prop:annihilators},
  \(
  \bigl(\symbvec{p}_{\mathrm{O}}, \BoldtraceNeu[\mathrm{B}](w), \BoldtraceNeu[\mathrm{F}](w)\bigr) \in \mathbb{X}_{\mathrm{N}}(\Sigma)
  \). 
  We emphasize that the expression of \(\symbvec{\A}_{\Gamma_{\mathrm{B}}}\)---which is the one arising from the Costabel FEM-BEM coupling---allows us to obtain \(\BoldtraceNeu[\mathrm{B}](w)\), and not simply \(\symbvec{p}_{\mathrm{B}}\). This plays a crucial role in the next point to recover Neumann transmission conditions on the skeleton, and finally prove that \(w \in H^1_{\mathrm{loc}}(\Delta, \overline{\Omega_{\mathrm{O}}^c})\).

  \ding{177} We now show that \(w\) belongs to \(H^1_{\mathrm{loc}}(\Delta, \overline{\Omega_{\mathrm{O}}^c})\), and solves the Helmholtz equation in \(\Omega_{\mathrm{O}}^c\). 
  Thanks to the definition of \(w\) in \(\Omega_{\beta}\) (\(\beta \in \setMI_{\mathrm{B}}\)), we know that \(w|_{\Omega_{\beta}}\) belongs to \(H^1_{\mathrm{loc}}(\Delta, \overline{\Omega_{\beta}})\) and Sommerfeld's radiation condition is satisfied if \(\Omega_{\mathrm{O}}\) is bounded, and that \(w|_{\Omega_{\beta}}\) is solution to the Helmholtz equation in \(\Omega_{\beta}\).
  With the definition of \(w\) in \(\Omega_{\alpha}\) (\(\alpha \in \setMI_{\mathrm{F}}\)) and what we have derived in~\ding{172}, we also know that \(w|_{\Omega_{\alpha}}\) belongs to \(H^1(\Delta, \Omega_{\alpha})\) and is solution to the Helmholtz equation in \(\Omega_{\alpha}\).
  Thus, we deduce that \(w\) belongs to \(L^2_{\mathrm{loc}}(\overline{\Omega_{\mathrm{O}}^c})\).
  
  Second, since
  \(
  (
  \symbvec{u}_{\mathrm{O}}, \symbvec{u}_{\mathrm{B}}, \symbvec{u}_{\mathrm{F}}
  )
  \in \mathbb{X}(\Omega_{\mathrm{O}}^c)\)
  by hypothesis and recalling that \(\mathbb{X}_{\mathrm{D}}(\Sigma) = \B(\mathbb{X}(\Omega_{\mathrm{O}}^c))\), we derive that
  \(
  \bigl(\symbvec{\phi}_{\mathrm{O}}, \symbvec{\phi}_{\mathrm{B}}, \BoldtraceDir[\mathrm{F}](\symbvec{u}_{\mathrm{F}})\bigr) = 
  \bigl(\symbvec{g}_{D}, \BoldtraceDir[\mathrm{B}](w), \BoldtraceDir[\mathrm{F}](w)\bigr) \in \mathbb{X}_{\mathrm{D}}(\Sigma)
  \),
  where we have replaced \(\symbvec{\phi}_{\mathrm{O}}\) and \(\symbvec{\phi}_{\mathrm{B}}\) with their values  respectively derived in \ding{174} and \ding{175}. Thus, applying Proposition~\ref{prop:neumann_annihilator} leads to \(w \in H^1_{\mathrm{loc}}(\overline{\Omega_{\mathrm{O}}^c})\) and \(\symbvec{g}_{D} = \BoldtraceDirC[\mathrm{O}](w)\). The (Dirichlet) boundary conditions on the impenetrable obstacles (if any) are then satisfied.
  
  Third, since we know that \(w|_{\Omega_{\alpha}} \in H^1(\Delta, \Omega_{\alpha})\), \(w|_{\Omega_{\beta}} \in H^1_{\mathrm{loc}}(\Delta, \overline{\Omega_{\beta}})\) and Equation~\eqref{eq:prop5_neumann_transm} holds, we can apply Proposition~\ref{prop:dirichlet_annihilator} for \(\Phi = \nabla w\). Thus, \(w \in H^1_{\mathrm{loc}}(\Delta, \overline{\Omega_{\mathrm{O}}^c})\) (and \(\symbvec{p}_{\mathrm{O}} = \BoldtraceNeuC[\mathrm{O}](w)\)).

  Eventually, we have shown that \(w\) solves Problem~\eqref{pbm:multi_domain_config}.
\end{proof}

Note that this reconstruction result holds also when Problem~\eqref{pbm:initialPb} (equivalently Problem~\eqref{pbm:multi_domain_config}), and thus the variational formulation~\eqref{pbm:single_domain_formulation}, is not well posed for certain wavenumbers, called resonances. It holds even for spurious resonances, that is, when~\eqref{pbm:single_domain_formulation} is not well-posed while~\eqref{pbm:initialPb} is.

The latter remarkable result is deteriorated if the \(\A_{\Gamma_{\beta}}\) arise from the Johnson-Nédélec FEM-BEM coupling instead of the Costabel one, for (\(\beta \in \setMI_{\mathrm{B}}\)). It holds if all the \(\Omega_{\beta}\) are bounded, but fails to be true if \(\Omega_{\mathrm{B},1}\) is unbounded and \(\kappa_{1}\) is a spurious resonance.
Indeed, at step \ding{176} of the previous proof we can only state that
\(
\bigl(\symbvec{p}_{\mathrm{O}}, \symbvec{p}_{\mathrm{B}}, \BoldtraceNeu[\mathrm{F}](w)\bigr) \in \mathbb{X}_{\mathrm{N}}(\Sigma)
\),
with \(\symbvec{p}_{\mathrm{B}} - \BoldtraceNeu[\mathrm{B}](w) \in \ker(\symbvec{\V}_{\mathrm{B}}) \cap \left(\symbvec{\Id}/2 + \symbvec{\Ktilde}_{\mathrm{B}}\right)\).
In~\cite[Lemma~1]{BoisneaultBonazzoliEtAl2026SRS}, we have proved that this intersection is reduced to the trivial element when \(\Omega_{\beta}\) is bounded, so \(\symbvec{p}_{\mathrm{B}} = \BoldtraceNeu[\mathrm{B}](w)\) and Neumann transmission conditions for \(w\) can be recovered.
Yet, these transmission conditions can not be derived if \(\Omega_{\mathrm{B},1}\) is unbounded and \(\kappa_{1}\) is a spurious resonance because we can not prove that \(\symbvec{p}_{\mathrm{B}} = \BoldtraceNeu[\mathrm{B}](w)\).
This issue can be seen as a lack of knowledge about \(p_{\mathrm{B},1}\). Indeed, if \(\kappa_{1}\) is a spurious resonance, then the first Calderón equation~\eqref{eq:calderon_1} does not imply the second Calderón equation~\eqref{eq:calderon_2}. Therefore, some ``information'' from the second Calderón equation is ``missing''. 
This interpretation also explains why the Costabel FEM-BEM coupling raises no issue, since in that case both Calderón equations are used to define the operators \(\A_{\Gamma_{\beta}}\).
Finally, note that what we have just explained indicates that, when the geometry contains an unbounded (homogeneous) subdomain, the Costabel coupling can be used for this subdomain, while the Johnson-Nédélec coupling can be used to deal with the other bounded (homogeneous) subdomains.

\thanks{This work is funded by the Inria program “Actions exploratoires” (OptiGPR3D).}


\bibliographystyle{crplain}
\bibliography{crmath_BBCM}

\end{document}